\documentclass{article}

\usepackage[a4paper,margin=1.4in]{geometry}

\usepackage{hyperref} 
\usepackage{url} 
\usepackage{lmodern} 
\usepackage{relsize}
\usepackage{graphicx}   	  
\usepackage{lastpage}         			                      
\usepackage{enumitem}

\usepackage[english]{babel}                 
\usepackage[utf8]{inputenc}

\usepackage{amsfonts}
\usepackage{amsmath}
\usepackage{amscd}
\usepackage{amssymb}
\usepackage{amsthm}
\usepackage{bbm}
\usepackage{dsfont}

\mathcode`l="8000
\begingroup
\makeatletter
\lccode`\~=`\l
\DeclareMathSymbol{\lsb@l}{\mathalpha}{letters}{`l}
\lowercase{\gdef~{\ifnum\the\mathgroup=\m@ne \ell \else \lsb@l \fi}}%
\endgroup

\newcommand{\C}{{\mathbb C}}

\renewcommand{\H}{{\mathbb H}}

\newcommand{\R}{{\mathbb R}}

\newcommand{\cA}{{\mathcal A}}

\newcommand{\cH}{{\mathcal H}}

\newcommand{\cJ}{{\mathcal J}}

\newcommand{\cL}{{\mathcal L}} 
\newcommand{\cM}{{\mathcal M}}

\newcommand{\cQ}{{\mathcal Q}}
\newcommand{\cR}{{\mathcal R}}
\newcommand{\cS}{{\mathcal S}}
\newcommand{\cT}{{\mathcal T}}

\newcommand{\fg}{{\mathfrak{g}}}
\newcommand{\fh}{{\mathfrak{h}}}

\newcommand{\SL}{{\textrm{SL}}}

\newcommand{\Diffex}{{\textrm{Diff}^{\textrm{ex}}}}

\newtheorem{TheoremABC}{Theorem}

\newtheorem {Theorem}{Theorem}[section]

\newtheorem {Lemma}[Theorem]{Lemma}
\newtheorem {Proposition}[Theorem]{Proposition}
\newtheorem {Definition}[Theorem]{Definition}

\theoremstyle{definition}
\newtheorem {Remark}[Theorem]{Remark}

\begin{document}

\title{The hyperkähler metric on the almost-Fuchsian moduli space}

\author{Samuel Trautwein\thanks{The author was partially supported by the Swiss National
Science Foundation (grant number  200021-156000).} 
}

\maketitle



\abstract{Donaldson \cite{Donaldson:2000} constructed a hyperkähler moduli space $\cM$ associated to a closed oriented surface $\Sigma$ with $\textrm{genus}(\Sigma) \geq 2$. This embeds naturally into the cotangent bundle $T^*\cT(\Sigma)$ of Teichmüller space or can be identified with the almost-Fuchsian moduli space associated to $\Sigma$. The latter is the moduli space of quasi-Fuchsian threefolds which contain a unique incompressible minimal surface with principal curvatures in $(-1,1)$.

Donaldson outlined various remarkable properties of this moduli space for which we provide complete proofs in this paper: On the cotangent-bundle of Teichmüller space, the hyperkähler structure on $\cM$ can be viewed as the Feix--Kaledin hyperkähler extension of the Weil--Petersson metric. The almost-Fuchsian moduli space embeds into the $\textrm{SL}(2,\C)$-representation variety of $\Sigma$ and the hyperkähler structure on $\cM$ extends the Goldman holomorphic symplectic structure. Here, the natural complex structure corresponds to the second complex structure in the first picture. Moreover, the area of the minimal surface in an almost-Fuchsian manifold provides a Kähler potential for the hyperkähler metric. 

The various identifications are obtained using the work of Uhlenbeck \cite{Uhlenbeck:1983} on germs of hyperbolic $3$-manifolds, an explicit map from $\cM$ to $\cT(\Sigma)\times \overline{\cT(\Sigma)}$ found by Hodge \cite{Hodge:PhD}, the simultaneous uniformization theorem of Bers \cite{Bers:1960}, and the theory of Higgs bundles introduced by Hitchin \cite{Hitchin:1987}.
}

\newpage

\setcounter{tocdepth}{2}
\tableofcontents

\newpage

\section{Introduction}

The concept of moment maps has a long history in symplectic geometry as a formalism to reduce the degrees of freedom in Hamiltonian mechanics \cite{MarsdenWeinstein:2001}. It has since then been realized, that many important differential geometric equations can be understood on a conceptual level from this viewpoint. The first example in this direction was given in the 80s by Atiyah--Bott \cite{AtBott:YangMillsEq} and asserts that the curvature provides a moment map for the action of the Gauge group on the space of connections. Most of the earlier developments in the 80s and 90s then generalized this construction to various gauge theoretical moduli spaces \cite{Donaldson:ASD4, Donaldson:1987, UYau:1986, Hitchin:1987, Simpson:1987, Bradlow:1991, Mundet:2000}. One notable exception is the observation of Quillen, Fujiki and Donaldson \cite{Fujiki:1990, Donaldson:1997}, that the scalar curvature can be understood as a moment map on the space of compatible almost complex structures on a symplectic manifold for the action of the Hamiltonian diffeomorphism group. A new impulse was given around the turn of the century by Donaldson \cite{Donaldson:2000, Donaldson:2000b}, where he described various frameworks and moment map constructions for actions of the diffeomorphism group.

This survey focuses mostly on the one article \cite{Donaldson:2000} of Donaldson. This may seem a bit odd for a survey article, but we will see that there is a wide range of mathematical areas and delicate constructions involved which are still highly relevant. One particular area which we have in mind are applications to higher Teichmüller theory, on which we elaborate a bit more below. Let $\Sigma$ be a closed oriented surface with $\textrm{genus}(\Sigma) \geq 2$. Donaldson shows, that the associated moduli space 
\begin{align} \label{INTRO:M0}
		\cM := \left\{ (g,\sigma) \in \textrm{Met}(\Sigma)\times Q(g)\,\left|\, \begin{array}{c}\bar{\partial} \sigma = 0,\, |\sigma|_g < 1,\\  K_g + |\sigma|_g^2 = -1 \end{array}\right.\right\} \bigg/ \text{Diff}_0(\Sigma)
\end{align}
carries a natural hyperkähler structure, where $Q(g)$ denotes the space of quadratic differentials for the complex structure $J_g$ determined by the metric $g$. There are various geometric viewpoints from which one can understand this moduli space. First, it embeds naturally as an open neighborhood of the zero section into the cotangent bundle of Teichmüller space $T^*\cT(\Sigma)$ and can be viewed as the Feix--Kaledin hyperkähler extension of the Weil--Peterson metric on Teichmüller space. Secondly, it embeds as an open subset into the quasi-Fuchsian moduli space $\mathcal{QF}(\Sigma)$ and its image parametrizes the class of almost-Fuchsian manifolds. These are quasi-Fuchsian $3$-manifolds which contain a unique incompressible minimal surface with principal curvatures in $(-1,1)$. The area of this minimal surface then provides a Kähler potential for the hyperkähler metric with respect to the standard complex structure. Thirdly, this moduli space can be viewed as an open subset of the smooth locus of the $\textrm{SL}(2,\R)$ representation variety of $\Sigma$ and one can give an explicit construction of the corresponding Higgs field under the non-abelian Hodge correspondence. An important feature of this construction is, that the identifications between the various models can be made rather explicitly.

While several aspects of this moduli space are well known to the experts, it seems that many of the interconnections between the different geometric models are lesser known. Here we think about the relation between the moment map picture and the almost-Fuchsian moduli space or the explicit construction of the non-abelian Hodge correspondence. Moreover, several details were stated at a somewhat conjectural level in the original work of Donaldson \cite{Donaldson:2000}. These can now be made more precise by using recent results on minimal surfaces in hyperbolic threefolds \cite{Wang:2012, HuangWang:2013, Hass:2015}. In order to provide a self-contained presentation, the article briefly surveys the well-known aspects of the theory and provides complete proofs of those aspects, which cannot be found in the existing literature. An extended version of this article, containing more details, can be found in Chapter 4 of the thesis \footnote{Note that the sign conventions there are slightly different from the ones used here.} \cite{ST:thesis}.

One of the main motivation in writing this survey is the belief, that one should be able to exploit Donaldson's constructions in the context of higher Teichmüller theories. The starting point for this would be a generalization of the setup to consider tuples consisting of a complex structure together with differentials of higher degree. Some (very) preliminary results in these directions are indicated in Chapter 5 of the author's thesis \cite{ST:thesis}. The main point is, however, that any such application necessarily goes beyond a simple application of the main results in Donaldson's paper and instead requires a thorough understanding of the details within the constructions. We hope that this survey helps to stimulate further research in this direction.

\paragraph{Donaldson's moment map framework}

Let $(M,\rho)$ be a closed $n$-dimensional manifold equipped with a volume form $\rho$, let $P \rightarrow M$ be its $\SL(n,\R)$-frame bundle and let $(X,\omega)$ be a symplectic manifold with a Hamiltonian $\SL(n,\R)$-action generated by a moment map $\mu: X \rightarrow \mathfrak{sl}(n,\R)^*$. 

Denote by $\cS(P,X)$ the space of section of the associated symplectic fibration $P(X) := P\times_{\SL(n,\R)}X$. This space carries a natural symplectic form defined by
	\begin{align} \label{intro:JQsymp}
		\underline{\omega}_s(\hat{s}_1, \hat{s}_2) := \int_M \omega_s(\hat{s}_1, \hat{s}_2) \rho
	\end{align}
for vertical vector fields $\hat{s}_1,\hat{s}_2 \in \Omega^0(M, s^*T^{vert} P(X))$. Let $\Diffex(M,\rho)$ denote the group of exact volume preserving diffeomorphisms. These are obtained by integrating exact divergence free vector fields and can be viewed as the Lie group corresponding to the Lie subalgebra 
	\begin{align} \textrm{Lie}\left(\Diffex(M,\rho) \right) = \left\{ v \in \textrm{Vect}(M)\,|\, \textrm{$\iota(v)\rho$ is exact} \right\}. \end{align}		
This space is isomorphic to $\Omega^{n-2}(M) / \textrm{ker}(d)$ and thus its dual space can formally be identified with the space of exact $2$-forms on $M$.

\begin{TheoremABC}[\textbf{Donaldson} \cite{Donaldson:2000}] \label{introJQ:ThmA}
Fix a torsion free $\SL(n,\R)$-connection $\nabla$ on $M$ and define $\underline{\mu}: \mathcal{S}(P,X) \rightarrow \Omega^2(M)$ by
	\begin{align}
		\underline{\mu}(s) := \omega(\nabla s \wedge \nabla s) - \langle \mu_s , R^{\nabla} \rangle - dc(\nabla \mu_s)
	\end{align}
where $\mu_s \in \Omega^0(M, \textrm{End}_0(TM)^*)$ is obtained by composing $s \in \cS(P,X)$ with the moment map on each fibre and $c(\nabla \mu_s) \in \Omega^1(M)$ is defined as the contraction $(\mu_s)_{j;i}^i$ of $\nabla \mu_s$. Then $\underline{\mu}$ is equivariant, independent of the connection $\nabla$ used to define it, takes values in the space of closed $2$-forms and satisfies the moment map equation.
\end{TheoremABC}

Note that $\underline{\mu}$ is not a moment map in the strict sense, since it takes only values in the space of closed 2-forms and not in the space of exact $2$-forms. Nevertheless, one can still make sense of the moment map equation, see Theorem \ref{fib.ThmM} and Remark \ref{fib.RmkM}.

\paragraph{The hyperkähler moduli space \texorpdfstring{$\cM$}{cM}}

In the following, let $(\Sigma,\rho)$ be a closed oriented $2$-dimensional surface with fixed area form $\rho \in \Omega^2(\Sigma)$ and assume $\textrm{genus}(\Sigma) \geq 2$. Consider as fibre the unit disc bundle $X \subset T^*\H$. This carries a unique $S^1\times\textrm{SL}(2,\R)$-invariant hyperkähler metric, which extends the hyperbolic metric along the zero section and blows up when approaching the boundary of the disc bundle (see Theorem \ref{ThmHK}). Moreover, a section $s \in \cS(P,X)$ can be identified with a pair $(J,\sigma)$ consisting of a complex structure $J$ and a quadratic differential $\sigma$ satisfying $|\sigma|_J < 1$. With this identification, the space $\cS(P,X)$ corresponds to
	\begin{align}
		 \cQ_1(\Sigma) := \left\{(J,\sigma)\, | \,J \in \cJ(\Sigma), \, \sigma \in \Omega^0(\Sigma, S^2(T^*\Sigma \otimes_J \C)), \, |\sigma|_J < 1 \right\}.
	\end{align} 	
The hyperkähler structure on $X$ then yields a hyperkähler structure on $\cQ_1(\Sigma)$ and Theorem \ref{introJQ:ThmA} asserts that there exists a hyperkähler moment map for the action of the Hamiltonian diffeomorphism group. 

\begin{TheoremABC}[\textbf{Donaldson} \cite{Donaldson:2000}] \label{introJQ:ThmB}
		The action of $\textrm{Ham}(\Sigma,\rho)$ on $\cQ_1(\Sigma)$ admits a hyperkähler moment map given by
			\begin{equation}
				\begin{gathered} 
				\underline{\mu}_1 (J,\sigma) = \frac{|\partial \sigma|^2 - |\bar{\partial} \sigma|^2}{\sqrt{1 - |\sigma|^2}} \rho - 2 \sqrt{1 - |\sigma|^2} K_J \rho - 2 \textbf{i} \bar{\partial} \partial 	\sqrt{1 - |\sigma|^2}  + 2c\rho \\
						\underline{\mu}_2(J,\sigma) + \textbf{i}\underline{\mu}_3(J,\sigma) = 2\textbf{i} \bar{\partial}  r(\bar{\partial} \sigma)
				\end{gathered}
			\end{equation}
		where $c := 2\pi(2 - 2\textrm{genus}(\Sigma))/\textrm{vol}(\Sigma,\rho)$ and $r: \Omega^{0,1}(\Sigma, S^2(T^*\Sigma \otimes_J \C) ) \rightarrow \Omega^{1,0}(\Sigma)$ is the contraction defined by the metric $\rho(\cdot,J\cdot)$. 
\end{TheoremABC}

We will prove in Theorem \ref{hk.ThmM} the slightly stronger statement that two of the moment maps extend to moment maps for the symplectomorphism group. The construction of the hyperkähler quotient $\cM$ is then obtained in two steps: First, it follows from general principles that the hyperkähler moment map above gives rise to a hyperkähler moduli space
\begin{align}
	\cM_0 := \underline{\mu}_1^{-1}(0)\cap \underline{\mu}_2^{-1}(0) \cap \underline{\mu}_3^{-1}(0) / \textrm{Ham}(\Sigma,\rho).
\end{align}
Secondly, this quotient carries a natural action of $H := \textrm{Symp}_0(\Sigma,\rho)/\textrm{Ham}(\Sigma,\rho)$ which admits moment maps $\mu^H_2$ and $\mu^H_3$ with respect to the second and third symplectic form. The vanishing locus of these two moment maps agree and we define accordingly
\begin{align}
	\cM_s := (\mu^H_2)^{-1}(0) / H = (\mu^H_3)^{-1}(0) / H.
\end{align}
Being a Marsden--Weinstein quotient, it follows directly that the second and third symplectic form descend to $\cM_s$. For the first symplectic structure, a different line of arguments is needed which depends on the fact that the $H$ orbits in the level sets are symplectic submanifolds. Finally, by using Moser isotopy and a suitable rescaling of the quadratic differential, it follows that $\cM_s$ can be identified with
\begin{align}
		\cM := \left\{ (g,\sigma) \in \textrm{Met}(\Sigma)\times Q(g)\,\left|\, \begin{array}{c}\bar{\partial} \sigma = 0,\, |\sigma|_g < 1,\\  K_g - \frac{c}{2}|\sigma|_g^2 = \frac{c}{2}\end{array}\right.\right\} \bigg/ \text{Diff}_0(\Sigma)
\end{align}
where $c := 2\pi(2 - 2\textrm{genus}(\Sigma))/\textrm{vol}(\Sigma,\rho)$ as above. After scaling the volume of $\Sigma$, we may assume that $c = -2$. Then, the moduli space $\cM$ takes the simpler form (\ref{INTRO:M0}).

\paragraph{Geometric models of the hyperkähler quotient}

An almost-Fuchsian hyperbolic $3$-manifold is a quasi-Fuchsian $3$-manifold which possesses an incompressible minimal surface with principal curvatures in $(-1,1)$. 

\begin{Remark}
The class of almost-Fuchsian manifolds is strictly smaller than the class of quasi-Fuchsian manifold: There are examples of quasi-Fuchsian manifolds which admit more than one minimal surface (see \cite{Wang:2012, HuangWang:2013, Hass:2015}). These cannot be almost-Fuchsian (see Lemma \ref{Lemma:Ymin}).
\end{Remark}

The isomorphism between $\cM$ and the space of almost-Fuchsian manifolds follows from Uhlenbeck's theory of minimal surfaces in hyperbolic $3$-manifolds \cite{Uhlenbeck:1983}.

\begin{TheoremABC}[\textbf{Uhlenbeck} \cite{Uhlenbeck:1983}] \label{IntroJQ:ThmMtoAF}
Let $g \in \textrm{Met}(\Sigma)$ and $\sigma \in Q(g)$ satisfy the equations $K_g + |\sigma|^2 = -1$, $\bar{\partial} \sigma = 0$, and $ |\sigma|_g < 1$. Then, there exists a unique hyperbolic manifold $Y$, together with an isometric embedding $\Sigma \rightarrow Y$ such that the second fundamental form of this embedding is given by $\textrm{Re}(\sigma)$. Moreover, $Y$ is diffeomorphic to $\Sigma\times\R$ and its hyperbolic metric is of almost-Fuchsian type.
\end{TheoremABC}

An explicit formula of this metric is given in Theorem \ref{Thm:MtoAF}. Moreover, there is a standard isomorphism $\mathcal{QF}(\Sigma) \cong \cT(\Sigma) \times \overline{\cT(\Sigma)}$ and the second complex structure on $\cM$ corresponds to the natural complex structure on the latter space. In this context, we have the following description of the hyperkähler metric.

\begin{TheoremABC}
Let $A: \mathcal{AF}(\Sigma) \rightarrow \R$ be the area functional, which assigns to an almost-Fuchsian manifold $Y$ the area of its unique minimal surface. Then 
		\begin{align}  2\textbf{i} \bar{\partial}_{J_2} \partial_{J_2} A = \underline{\omega}_2 .\end{align}
Hence $A$ provides a Kähler potential with respect to the natural complex structure on $\mathcal{AF}(\Sigma)$ which agrees (up to sign) with the second complex structure on $\cM$.
\end{TheoremABC}

Every complete hyperbolic $3$-manifold $Y$ arises as a quotient of the hyperbolic space $\H^3$ and thus gives rise to a representation of $\rho: \pi_1(Y) \rightarrow \textrm{PSL}(2,\C)$, which is well-defined up to conjugation by an element in $\textrm{PSL}(2,\C)$.
This defines a natural embedding of the quasi-Fuchsian moduli space into the representation variety $\cR_{\textrm{PSL}(2,\C)}(\Sigma)$. Moreover, this embedding can be lifted to the representation variety of $\SL(2,\C)$, see e.g. \cite{Culler:1986}. The corresponding embedding of $\cM$ into $\cR_{\textrm{SL}(2,\C)}(\Sigma)$ can be described explicitly in terms of the non-abelian Hodge theory and Higgs bundles \cite{Hitchin:1987}. Let $g \in \textrm{Met}(\Sigma)$ and $\sigma \in Q(g)$ be a quadratic differential for the induced conformal structure. Choose a complex line bundle $L \rightarrow \Sigma$ with $L^2 = T\Sigma$ and define $E = L\oplus L^{-1}$. The Levi-Civita connection for $g$ induces a unique $U(1)$-connection $a \in \cA(L)$. Then consider the pair
	\begin{align} \label{introJQ:Hitchin}
		A = \begin{pmatrix}a & \frac{\bar{\sigma}}{2} \\ -\frac{\sigma}{2} & -a \end{pmatrix}\in \cA(E) \quad \textrm{and} \quad \phi = \frac{1}{2} \left(\begin{array}{cc} 0 & \textbf{1} \\ 0 & 0 \end{array}\right) \in \Omega^{1,0}(\text{End}(E))
	\end{align}
where $\sigma \in \Omega^{1,0}(L^{-2}) = \Omega^{1,0}(\text{Hom}(L,L^{-1}))$ and $\textbf{1} \in \Omega^0(\text{End}(T\Sigma)) = \Omega^{1,0}(L^2) = \Omega^{1,0}(\text{Hom}(L^{-1},L))$.

\begin{TheoremABC} \label{ThmE}
Let $g \in \textrm{Met}(\Sigma)$ and $\sigma \in Q(g)$ satisfy the equations $K_g + |\sigma|^2 = -1$, $\bar{\partial} \sigma = 0$, and $ |\sigma|_g < 1$. The corresponding pair $(A, \phi)$ defined by (\ref{introJQ:Hitchin}) then satisfies the Hitchin equation
$\bar{\partial}_A \phi = 0$, $F_A + [\phi \wedge \phi^*]= 0$ and $B := A + \phi + \phi^* $ is a flat $\textrm{SL}(2,\C)$-connection. The holonomy representation $\rho_B: \pi_1(\Sigma) \rightarrow \textrm{SL}(2,\C)$ agrees up to conjugation with the representation associated to the almost-Fuchsian associated to the pair $(g,\sigma)$.
\end{TheoremABC}

The cotangent bundle of Teichmüller space can be identified with the space
	$$T^*\cT(\Sigma) := \left\{ (J,\sigma)\,|\, J \in \cJ(\Sigma),\, \sigma \in Q(J), \, \bar{\partial}_J \sigma = 0 \right\} \! \big/ \, \textrm{Diff}_0(M).$$
It follows from a standard application of the continuation method that the natural map from $\cM$ into $T^*\cT(\Sigma)$ is an embedding and moreover the following holds.

\begin{TheoremABC} \label{ThmF}
[\textbf{Donaldson} \cite{Donaldson:2000}, \textbf{Hodge} \cite{Hodge:PhD}]
The hyperkähler structure of $\cM$ along the image of its embedding into $T^*\cT(\Sigma)$ agrees with the Feix--Kaledin hyperkähler extension of the Weil--Petersson metric on $\cT(\Sigma)$.
\end{TheoremABC}

Taubes \cite{Taubes:2004} investigated extensions of the maps in Theorem \ref{ThmE} and Theorem \ref{ThmF} to the larger moduli space which one obtains by omitting the constraint $|\sigma|_g < 1$ in the definition of $\cM$. See Remark \ref{Rmk:Taubes1} and Remark \ref{Rmk:Taubes2} for more details.

\subsubsection*{Acknowledgment}

I would like to thank my supervisor D. A. Salamon for many helpful discussions throughout the process of writing this paper. I am also indebted to the referee for his careful reading and excellent feedback.

\section{Donaldson's moment map}

Let $(M,\rho)$ be a closed oriented $n$-dimensional manifold with fixed volume form $\rho$ and let $P \rightarrow M$ be its $\SL(n,\R)$-frame bundle which is defined by
	$$P := \{ (z,\theta)\,|\, z \in M,\, \theta \in \textrm{Hom}(\mathbb{R}^n, T_z M),\, \theta^* \rho_z = \text{dvol}_{\mathbb{R}^n} \}.$$
Let $(X,\omega)$ be a symplectic manifold with Hamiltonian $\textrm{SL}(n,\R)$-action induced by an equivariant moment map $\mu: X \rightarrow \mathfrak{sl}^*(n, \mathbb{R})$ and consider the associated bundle
		$$P(X) := P\times_{\SL(n,\R)} X := (P \times X)/ \SL(n,\R) $$ 
where $\SL(n,\R)$ acts diagonally. Denote by $\mathcal{S}(P,X)$ its space of sections. 
The first two subsections summarize the necessary background on symplectic fibrations and the action of the diffeomorphism group. The main result of this section is Theorem \ref{fib.ThmM}, which contains a precise formulation of Donaldson's moment map picture for the action of the group $\textrm{Diff}_{ex}(M,\rho)$ of exact volume preserving diffeomorphism on $\cS(P,X)$.

\subsection{Symplectic fibrations}

The space $\cS(P,X)$ is formally an infinite dimensional symplectic manifold, where the tangent space at $s \in \cS(P,X)$ is given by the space of vertical vector fields along $s$
			$$T_s \cS(P,X) = \Omega^0(M, s^* T^{vert} P(X)).$$
The symplectic form on $X$ induces a symplectic structure on the vertical tangent bundle $T^{vert} P(X)$ and then by integration on $\mathcal{S}(P,X)$:
	\begin{align}
		\underline{\omega}_s(\hat{s}_1, \hat{s}_2) := \int_M \omega (\hat{s}_1, \hat{s}_2) \rho
	\end{align}
for $\hat{s}_1, \hat{s}_2 \in \Omega^0(M, s^* T^{vert} P(X))$. Moreover, a connection $A \in \cA(P)$ induces a covariant derivative $\nabla: \cS(P,X) \rightarrow \Omega^1(M, s^*T^{vert} P(X))$ defined by
		\begin{equation} \label{fib.cov} 
				\nabla_{\hat{p}} s(p)  = ds(p)\hat{p} + L_{s(p)} A_p(\hat{p}) = ds(p) \hat{p}^{hor}. 
		\end{equation}
In this formula, we lift $s \in \cS(P,X)$ to an equivariant map $s: P \rightarrow X$, denote by $L_x: \mathfrak{sl}(n,\R) \rightarrow T_x X$ the infinitesimal action, and let $\hat{p}^{hor} := \hat{p} - p \cdot A_p(\hat{p})$ be the horizontal component of a tangent vector $\hat{p} \in T_p P$. Moreover, we identify $\Omega^1(M, s^*T^{vert} P(X))$ in the usual way with the space of horizontal and equivariant $1$-forms on $P$ taking values in $s^*TX$.

\subsection{Action of the diffeomorphism group}

The group $\text{Diff}(M,\rho)$ of volume preserving diffeomorphisms can be viewed as infinite dimensional Lie group with Lie algebra
					$$\text{Lie}\left( \text{Diff}(M,\rho) \right) = \left\{ v \in \text{Vect}(M)\,|\, d\iota(v) \rho = 0 \right\}.$$
Every $\phi \in \text{Diff}(M,\rho)$ lifts naturally to an equivariant diffeomorphism $\tilde{\phi} : P \rightarrow P$ defined by
		$$\tilde{\phi}(z,\theta) := (\phi(z), d\phi(z)\circ\theta)$$
for $z \in M$ and $\theta \in \textrm{Hom}(\R^n, T_zM)$. This induces a natural action 
	$$\text{Diff}(M,\rho) \times \mathcal{S}(P,X) \rightarrow \mathcal{S}(P,X), \qquad \phi^*s := s\circ\tilde{\phi}$$
where we view elements of $\cS(P,X)$ as equivariant maps $s: P \rightarrow X$.

There is a one-to-one correspondence between connections $A \in \cA(P)$ and $\SL(n,\R)$-connections $\nabla$ on $TM$. For the calculation of the infinitesimal action it is useful to adopt the later point of view and to choose a torsion-free $\SL(n,\R)$-connections on $TM$.

\begin{Lemma} \label{fib.LemmaInf1}
Choose a torsion-free $\text{SL}(n,\mathbb{R})$-connection $\nabla$ on $TM$ and denote by $A \in \cA(P)$ the corresponding connection $1$-form on $P$. Let $v \in \text{Vect}(M)$ with $d\iota(v) \rho = 0$ be given and denotes its flow by $\phi_v^t \in \text{Diff}(M,\rho)$.
		\begin{enumerate}
		\item The infinitesimal action of $v$ on $P$ is defined as
			$$ \mathcal{L}_v (z, \theta) := \left.\frac{d}{dt}\right|_{t=0} ( \phi_v^{t}(z),  d\phi_v^{t}(z)\circ \theta ) \in T_{(z,\theta)} P$$
		and satisfies for all $(z,\theta) \in P$
			\begin{align} \label{fib.infeq1} d \pi(z,\theta) \mathcal{L}_v (z,\theta) = v(z), \qquad A_{(z,\theta)}(\mathcal{L}_v (z,\theta)) = \theta^{-1} \left(\nabla_{\theta(\cdot)} v\right)(z). \end{align}
		where $\pi: P \rightarrow M$ denotes the projection map.
		
			\item Denote by $\nabla v : P \rightarrow \mathfrak{sl}(n,\R)$ the map $(z,\theta) \mapsto  \theta^{-1} \left(\nabla_{\theta(\cdot)} v\right)(z)$ . Then
					\begin{align} \label{fib.infeq2} \mathcal{L}_v s := \left.\frac{d}{dt}\right|_{t=0} (\phi_v^t)^* s = \nabla_v s - L_s (\nabla v). \end{align}
			\end{enumerate}			
\end{Lemma}

\begin{proof}
The first part of (\ref{fib.infeq1}) follows from differentiating $\pi ( \phi_v^{t}(z),  d\phi_v^{t}(z)\circ \theta ) = \phi_v^{t}(z)$. For the second part, we use that $\nabla$ is a torsion-free connection corresponding to $A \in \cA(P)$ and hence
\begin{align*}
	A_{(z,\theta)}(\mathcal{L}_v(z,\theta))\xi &= \theta^{-1} \left.\nabla_t  d \phi_v^t(z) \theta(\xi) \right|_{t=0} 
																									 = \theta^{-1}\nabla_{\theta(\xi)} \left.\partial_t \phi_v^t(z) \right|_{t=0}
																									 = \theta^{-1}\nabla_{\theta(\xi)} v(z)
\end{align*}																						
for every $\xi \in \mathbb{R}^n$. This completes the proof of (\ref{fib.infeq1}). 

Next, let $s : P \rightarrow X$ be an equivariant map. By the chain rule and (\ref{fib.cov}), it follows
			$$(\mathcal{L}_v s)(p) = ds(p)[ \mathcal{L}_v (p)] = \nabla s(p)[\mathcal{L}_v (p)] - L_{(s(p))} A_p(\mathcal{L}_v (p))$$
for every $p \in P$. Inserting (\ref{fib.infeq1}) into this equation yields (\ref{fib.infeq2}).
\end{proof}

A diffeomorphism $\phi \in \textrm{Diff}(M,\rho)$ is called exact, if there exists an isotopy 
	$$\phi : [0,1] \rightarrow \textrm{Diff}(M,\rho), \quad t \mapsto \phi_t$$ 
with $\phi_0 = \mathds{1}$ and $\phi_1 = \phi$, and there exists a smooth map $v: [0,1] \rightarrow \textrm{Vect}(M)$ such that $\partial_t \phi_t = v_t\circ \phi_t$ and $\iota(v_t)\rho$ is exact for all $t \in [0,1]$. The group $\textrm{Diff}_{ex}(M,\rho)$ of all exact diffeomrophism is the subgroup corresponding to the Lie subalgebra
					$$\text{Lie}\left( \textrm{Diff}_{ex}(M,\rho) \right) = \left\{ v \in \text{Vect}(M)\,|\,  \iota(v) \rho \in d \Omega^{n-2}(M) \right\}.$$
The dual space of the Lie algebra can be identified with the space of exact $2$-forms on $M$ using the pairing
		$$\Omega^2_{ex}(M) \times \textrm{Lie}(\textrm{Diff}_{ex}(M,\rho)) \rightarrow \R, \qquad	\langle \tau , v \rangle := \int_M \tau \wedge \alpha_v$$
where $\alpha_v \in \Omega^{n-2}(M)$ satisfies $d \alpha_v = \iota(v) \rho$. Note that this pairing is well-defined and does not depend on the choice of the primitive $\alpha_v$ by Stokes theorem.

\subsection{Donaldson's moment map}

Fix a torsion free $\textrm{SL}(n,\R)$-connection on $TM$ and denote by $A \in \cA(P)$ the corresponding connection $1$-form on $P$. There exists a natural isomorphism $\textrm{ad}(P) \cong \textrm{End}_0(TM)$ and we denote by $R^{\nabla} \in \Omega^2(M, \text{End}_0(TM))$ the curvature of this connection. We introduce the three terms of the moment map in the following. 

First, define $\omega(\nabla s \wedge \nabla s) \in \Omega^2(M)$ by coupling the exterior product on $M$ with the symplectic from on $T^{vert}P(X)$:
			\begin{align} \label{fib.comp1} \omega(\nabla s \wedge \nabla s): \,TM \times TM \rightarrow \mathbb{R}, \qquad (u,v) \mapsto \omega(\nabla_u s , \nabla_v s) \end{align}
Second, view $s \in \mathcal{S}(P,X)$ as an equivariant map $s:P \rightarrow X$. The composition $\mu\circ s: P \rightarrow \textrm{sl}(n,\R)^*$ is equivariant and thus descends to a section $\mu_s \in \Omega^0(M, \text{End}_0(TM)^*)$. The duality pairing gives then rise to the two form
							\begin{align} \label{fib.comp2} \langle \mu_s , R \rangle \in \Omega^2(M). \end{align}
Third, define $c(\nabla \mu_s) \in \Omega^1(M)$ as the contraction $(\mu_s)_{j;i}^i$ of the covariant derivative $\nabla \mu_s \in \Omega^1(M, \text{End}_0(TM)^*)$, which is obtained as the trace over the first and third index. This is explicitly defined by
					\begin{align} \label{fib.comp3}c(\nabla \mu_s) \in \Omega^1(M), \qquad \hat{m} \mapsto \sum_{i=1}^n \left\langle \nabla_{e_i} \mu_s , (\hat{m} \otimes e^i)_0 \right\rangle \end{align}
with respect to any local frame.

\begin{Theorem}[\textbf{Donaldson's moment map}] \label{fib.ThmM}
Let $\nabla$ be a torsion-free $\text{SL}(n,\mathbb{R})$ on $TM$ and define $\underline{\mu}: \mathcal{S}(P,X) \rightarrow \Omega^2(M)$ by
		\begin{align} \label{fib.Meq1}
			\underline{\mu}(s) := \omega(\nabla s \wedge \nabla s) - \langle \mu_s , R \rangle - dc(\nabla \mu_s)
		\end{align}
where the expression on the right hand side are defined in (\ref{fib.comp1}), (\ref{fib.comp2}) and (\ref{fib.comp3}).
				\begin{enumerate}
					\item $\underline{\mu}(s)$ is closed and independent of the choice of the connection $\nabla$.
					\item $\underline{\mu}$ satisfies the naturality condition $\underline{\mu}(\phi^*s) = \phi^* \underline{\mu}(s)$ for every $s \in \cS(P,X)$ and every $\phi \in \textrm{Diff}(M,\rho)$.
					\item Let $v \in \textrm{Vect}(M)$ be an exact divergence free vector field and choose a primitive $\alpha_v \in \Omega^{n-2}(M)$ with $d\alpha_v = \iota(v) \rho$. The derivative of the map
									 \begin{align} \label{fib.Meq2} \mathcal{S}(P,X) \rightarrow \mathbb{R}, \qquad s \mapsto \int_M \underline{\mu}(s) \wedge \alpha_v \end{align}
					is the map $T_s\mathcal{S}(P,X) \rightarrow \mathbb{R}$ defined  by 
									\begin{align} \label{fib.Meq3}  \hat{s} \mapsto \underline{\omega}(\hat{s}, \mathcal{L}_v s ) = \int_M \omega(- \nabla_v s + L_s \nabla v, \hat{s} ) \rho \end{align}
				\end{enumerate}
\end{Theorem}

\begin{proof}
This is Theorem 9 in \cite{Donaldson:2000} and a detailed exposition of the proof can be found in \cite{ST:thesis}, Theorem 4.2.4. (Note that the signs in (\ref{fib.Meq1}) are slightly different than the ones in \cite{Donaldson:2000} due to different conventions.)
\end{proof}

\begin{Remark} \label{fib.RmkM}
The map $\underline{\mu}$ is not a moment map in the strict sense, since it takes values in the space of closed 2-forms. Let $v \in \textrm{Vect}(M)$ be an exact divergence free vector field and choose $\alpha_v \in \Omega^{n-2}(M)$ with $d\alpha_v = \iota(v)\rho$. Then
				$$\langle \underline{\mu}(s), v \rangle = \int_M \underline{\mu}(s)\wedge \alpha_v$$
depends on the choice of the primitive $\alpha_v$. Different choices for $\alpha_v$ change the pairing only by a constant, and so its derivative is well-defined and independent of any choices. The equations (\ref{fib.Meq3}) and (\ref{fib.Meq2}) show that $\underline{\mu}$ satisfies the moment map equation.
\end{Remark}

\section{Complex structures and quadratic forms}

The main application of Theorem \ref{fib.ThmM} considered in this paper arises when one takes the hyperbolic plane and its cotangent bundle as fibres. The purpose of this section is to recall some fundamental properties of these spaces and to establish our notation. In particular, the hyperbolic plane can be identified with the space of linear complex structures on $\R^2$ and its cotangent bundle can be identified with pairs $(J,q)$ consisting of a complex structure and a complex quadratic form.

\subsection{The space of complex structures on the plane}

Let $\mathbb{H} := \{ z \in \mathbb{C}\,|\, \text{Im}(z) > 0\}$ denote the upper half plane. It has a canoncial complex structure and we endow it with the hyperbolic metric and volume form
						$$g_{\mathbb{H}}(x,y) = \frac{dx^2 +  dy^2}{y^2}, \qquad \omega_{\mathbb{H}}(x,y) = \frac{dx\wedge dy}{y^2}.$$
The group $\text{SL}(2,\mathbb{R})$ acts on $\mathbb{H}$ by Möbius transformations
				\begin{align} \label{action1} \text{SL}(2,\mathbb{R})\times\mathbb{H} \mapsto \mathbb{H}, \qquad \left(\begin{array}{cc} a & b \\ c & d \end{array}\right) z := \frac{az + b}{cz+d} .\end{align}
Every Möbius transformation is a Kähler isometry of $\H$. Since this action is transitive with stabilizer $\textrm{SO}(2)$ at $\textbf{i}$, it gives rise to an identification $\mathbb{H} \cong \text{SL}(2,\mathbb{R})/\text{SO}(2)$.\\

The space of linear complex structures on $\mathbb{R}^2$, compatible with the standard orientation, is given by
				$$\mathcal{J}(\mathbb{R}^2) := \left\{ J \in \text{End}(\mathbb{R}^2)\,|\, J^2 = - \mathds{1}, \, \det(v, Jv) > 0 \right\}$$
where the condition does not depend on the choice of $v \in \R^2\backslash\{0\}$. The space $\mathcal{J}(\mathbb{R}^2)$ is a Kähler manifold, where the complex structure on the tangent space
					$$T_J \mathcal{J}(\mathbb{R}^2) = \left\{ \hat{J} \in \text{End}(\mathbb{R}^2)\,|\, J \hat{J} + \hat{J} J = 0 \right\}$$
is given by $\hat{J} \mapsto J \hat{J}$ and the metric and symplectic form are
	\begin{gather}\label{cJ:structures}
			\omega_\mathcal{J}(\hat{J}_1, \hat{J}_2) = -\frac{1}{2}\text{tr}\left(\hat{J}_1 J \hat{J}_2\right) , \qquad
					g_\mathcal{J}(\hat{J}_1, \hat{J}_2) = \frac{1}{2}\text{tr}\left(\hat{J}_1 \hat{J}_2\right) 
	\end{gather}
The group $\text{SL}(2,\mathbb{R})$ acts on $\cJ(\mathbb{R}^2)$ by conjugation $\Psi_* J = \Psi J \Psi^{-1}$. This action preserves the Kähler structure on $\cJ(\R^2)$. Moreover, it is transitive with stabilizer $\textrm{SO}(2)$ at the standard complex structure 
				$$J_0 := \left(\begin{array}{cc} 0 & -1 \\ 1 & 0 \end{array}\right)$$ 
This gives rise to an identification $\mathcal{J}(\mathbb{R}^2) \cong \text{SL}(2,\mathbb{R})/\textrm{SO}(2)$. \\

Define $j: \mathbb{H} \rightarrow  \mathcal{J}(\mathbb{R}^2)$ as the composition $\H \cong \text{SL}(2,\mathbb{R})/\textrm{SO}(2) \cong \mathcal{J}(\mathbb{R}^2)$. It follows from our discussion above that this is the unique $\SL(2,\mathbb{R})$-equivariant map which satisfies $j(\textbf{i}) = J_0$. A short calculation yields the formula 
		\begin{align} \label{Hcpx}
					j: \mathbb{H} \rightarrow  \mathcal{J}(\mathbb{R}^2), \qquad j(x + \textbf{i}y) := \left( \begin{array}{cc} \frac{x}{y} & - \frac{x^2 + y^2}{y} \\ \frac{1}{y} & -\frac{x}{y} \end{array} \right)
		\end{align}
and one verifies that this is an antiholomorphic and antisymplectic isometry.	

\begin{Remark}
One is tempted to defined the complex structure on $\mathcal{J}(\mathbb{R}^2)$ with a different sign in order to make $j$ a Kähler isometry. However, this would be in conflict with standard sign conventions in Teichmüller theory and introduce unpleasant signs at other places.
\end{Remark}

\begin{Lemma} \label{LemmaMM0} The action of $\textrm{SL}(2,\R)$ on $\cJ(\R^2)$ and $\H$ is Hamiltonian and generated by the equivariant moment maps 
	\begin{gather*}
		\mu_{\cJ}: \mathcal{J}(\mathbb{R}^2) \rightarrow \mathfrak{sl}^*(2,\mathbb{R}), \qquad \langle \mu_{\cJ}(J), \xi \rangle := \text{tr}(J \xi), \\
		\mu_{\H}: \mathbb{H} \rightarrow \mathfrak{sl}^*(2,\mathbb{R}), \qquad \langle \mu_{\H}(z), \xi \rangle := -\text{tr}(j(z) \xi)
 \end{gather*}
for $\xi \in \mathfrak{sl}(2,\R)$ and $j: \mathbb{H} \rightarrow \mathcal{J}(\mathbb{R}^2)$ defined by (\ref{Hcpx}).
\end{Lemma}

\begin{proof}
Let $J \in \mathcal{J}(\mathbb{R}^2)$, $\hat{J} \in T_J \mathcal{J}(\mathbb{R}^2)$ and $\xi \in \mathfrak{sl}(2,\mathbb{R})$. The infinitesimal action of $\xi$ at $J$ is given by $L_J \xi = [\xi,J]$ and therefore
	\begin{align*}
		\omega_{\mathcal{J}}\left( L_J \xi, \hat{J} \right) 
				= -\frac{1}{2} \text{tr}\left(  [\xi, J] J \hat{J}\right) 
				= \frac{1}{2}\left( \text{tr}\left( \xi\hat{J} \right) + \text{tr}\left(J\xi J \hat{J} \right) \right)
				= \text{tr}(\hat{J}\xi)
		\end{align*}
This proves the first part of the lemma. The second part follows from this, since $j$ is equivariant and antisymplectic.
\end{proof}				

Denote by $\omega_0 = dx\wedge dy$ the standard area form on $\R^2$. Every $J \in \mathcal{J}(\mathbb{R}^2)$ defines a hermitian form on $(\R^2,J)$ defined by
			\begin{align} \label{HJ.Qeqh}
						h_J: \mathbb{R}^2\times\mathbb{R}^2 \rightarrow \mathbb{C}, \qquad h_J(\cdot, \cdot) := \omega_0(\cdot, J \cdot) + \textbf{i}\omega_0(\cdot, \cdot)
			\end{align}			
This is complex anti-linear in the first coordinate and complex linear in the second coordinate with respect to $J$. A direct computation shows that $h_{j(z)}$ has the matrix representation
							\begin{align} \label{HJ.hJeq0} h_{j(z)}(v,w) = v^t\,\frac{1}{\text{Im}(z)} \left( \begin{array}{cc} 1 & - \bar{z} \\ -z & |z|^2 \end{array} \right)\, w. \end{align}
for $v,w \in \R^2$.

\subsection{Complex quadratic forms}

For $J \in \mathcal{J}(\mathbb{R}^2)$ we denote the space of complex quadratic forms on $(\R^2,J)$ by 
\begin{align*} 
				Q(J) &:= \{ q : \mathbb{R}^2 \times \mathbb{R}^2 \rightarrow \mathbb{C}\,|\, \textrm{$(J,\textbf{i})$-complex bilinear and symmetric} \}.
\end{align*}
This carries the complex structure $q \mapsto \textbf{i}q$ and the hermitian structure
\begin{equation} \label{Q(J):structures}
		g_Q(q_1, q_2) := \text{Re}\left(\frac{\overline{q_1(v,v)}q_2(v,v)}{h_J(v,v)^2}\right), \quad
		\omega_Q(q_1, q_2) := \text{Im}\left(\frac{\overline{q_1(v,v)}q_2(v,v)}{h_J(v,v)^2}\right)
\end{equation}
where $h_J$ is defined by (\ref{HJ.Qeqh}) and neither expression depends on $v \in \mathbb{R}^2\backslash\{0\}$.

\subsubsection{Identification with tangent vectors}
Consider the map $T_J \mathcal{J}(\mathbb{R}^2) \rightarrow Q(J)$ defined by
	\begin{align} \label{HJ.QeqT}
		\hat{J} \mapsto q_{(J,\hat{J})} := h_J(\hat{J}\cdot , \cdot).
	\end{align}			
We show in the next lemma that this map is a complex antilinear isometry with respect to the structures defined in (\ref{cJ:structures}) and (\ref{Q(J):structures}).

\begin{Lemma}[\textbf{Quadratic forms and tangent vectors}]
\hspace{2em}
		\begin{enumerate}
			\item For $J \in \mathcal{J}(\mathbb{R}^2)$ and $\hat{J} \in T_J \mathcal{J}(\mathbb{R}^2)$ it holds
									\begin{align} \label{HJ.Qsym} h_J(\hat{J}v, w) = h_J(\hat{J}w, v) \qquad \textrm{for all $v,w \in \mathbb{R}^2$}.\end{align}
			In particular, $q_{(J,\hat{J})} := h_J(\hat{J} \cdot, \cdot) \in Q(J)$.
			
			\item For every $J \in \mathcal{J}(\mathbb{R}^2)$ the map (\ref{HJ.QeqT}) is a complex antilinear isomorphism with respect to the structures defined in (\ref{cJ:structures}) and (\ref{Q(J):structures}).
		
			\item The collection of maps (\ref{HJ.QeqT}) is $\SL(2,\R)$-equivariant in the following sense: Let $J \in \mathcal{J}(\mathbb{R}^2)$, $\hat{J} \in T_J \mathcal{J}(\mathbb{R}^2)$, and $\Psi \in \text{SL}(2,\mathbb{R})$, then
					$$q_{\Psi_*(J,\hat{J})}(v,w) = q_{(J,\hat{J})} (\Psi^{-1} v, \Psi^{-1} w) \qquad \textrm{for all $v,w \in \R^2$.}$$
						
		\end{enumerate} 
\end{Lemma}

\begin{proof}
Differentiating the equation $\omega_0(Jv, Jw) = \omega_0(v,w)$ it follows 
	$$\omega_0(\hat{J}v, Jw) + \omega_0(Jv, \hat{J} w) = 0.$$ 
Hence $\omega_0(\hat{J}v, Jw) = \omega_0(v, J\hat{J} w)$ shows that $\hat{J}$ is self-adjoint with respect to the inner product $\omega_0(\cdot, J \cdot)$ and $\omega_0(\hat{J} v,w) = - \omega_0(v, \hat{J}w)$. Then follows
 $$h_J(\hat{J}v, w) = \omega_0(\hat{J}v, Jw) + \textbf{i}\omega_0(\hat{J}v, w) 
										= \omega_0(v, J \hat{J}w) + \textbf{i}\omega_0(\hat{J}w, v)  = h_J(\hat{J}w, v)$$
This completes the proof of (\ref{HJ.Qsym}).		
			
For the second part, it follows from (\ref{HJ.Qsym}) that
	\begin{align*}
		||\hat{J}||^2 &= \frac{1}{2} \left(\frac{h_J(\hat{J}^2 v, v)}{h_J(v,v)} + \frac{h_J(\hat{J}^2 Jv, Jv)}{h_J(v,v)} \right) 
									= \frac{h_J(\hat{J} v, \hat{J} v)}{h_J(v,v)} 
									= \frac{ |h_J(v, \hat{J} v)|^2 }{h_J(v,v)^2} 
									= ||q_{\hat{J}}||^2									
	\end{align*}
where we used in the penultimate equation that $(\mathbb{R}^2,J)$ is complex one-dimensional and hence $|h_J(v,\hat{J}v)|^2 = h_J(v,v) h_J(\hat{J}v, \hat{J}v)$. Hence (\ref{HJ.QeqT}) is an isometry. Since it is clearly complex antilinear, this completes the prove of the second part.

Finally, let $\Psi \in \text{SL}(2,\mathbb{R})$ be given and compute
	\begin{align*}
		q_{\Psi_*(J, \hat{J})} &= h_{\Psi J \Psi^{-1}}( \Psi \hat{J} \Psi^{-1} \cdot, \cdot) 
														= \omega_0( \Psi \hat{J} \Psi^{-1} \cdot,\Psi J \Psi^{-1} \cdot) + \omega_0( \Psi \hat{J} \Psi^{-1} \cdot, \cdot ) \\
													 &= \omega_0( \hat{J} \Psi^{-1} \cdot, J \Psi^{-1} \cdot) + \omega_0( \hat{J} \Psi^{-1} \cdot, \Psi^{-1}\cdot ) 
													 = q_{(J,\hat{J})} (\Psi^{-1} \cdot, \Psi^{-1} \cdot).
	\end{align*}	
This proves equivariance and the lemma.
\end{proof}

\subsubsection{Identification with covectors}

The Riemannian metric on $\mathcal{J}(\mathbb{R}^2)$ defines a complex antilinear isomorphism of the tangent bundle and cotangent bundle of $\mathcal{J}(\mathbb{R}^2)$. This is given by
		\begin{align} \label{HJ.Qdual}
				T_J \mathcal{J}(\mathbb{R}^2)  \rightarrow T_J^* \mathcal{J}(\mathbb{R}^2), \qquad \hat{J} \mapsto \left( \hat{J}' \mapsto \frac{1}{2}\text{tr}\left( \hat{J} \hat{J}' \right)\right).
		\end{align}		
Combining this map with (\ref{HJ.QeqT}) and the derivative of (\ref{Hcpx}), gives rise to an identification between the cotangent bundle $T^*\H$ the hyperbolic plane and the bundle of quadratic forms over $\cJ(\R^2)$:
	$$(j,q): T^*\mathbb{H} \cong \cQ(\R^2) := \left\{ (J,q) \,|\, J \in \cJ(\R^2), \, q \in Q(J) \right\}.$$ 
The next lemma asserts that this map is an antiholomorphic $\SL(2,\mathbb{R})$-equivariant diffeomorphism which restricts to isometries along the fibres.

\begin{Lemma}[\textbf{Quadratic forms and covectors}] \label{HJ.QLemma1}
Define the map 
	$$(j, q) : T^*\mathbb{H} \rightarrow \cQ(\R^2) \subset \mathcal{J}(\mathbb{R}^2)\times\text{Hom}(\mathbb{R}^2\otimes \mathbb{R}^2 , \mathbb{C})$$
for $z = x +\textbf{i}y \in \H$ and $w \in \C$ by 	
\begin{align} \label{HJ.Qq}
		j(z,w) := j(z) := \left( \begin{array}{cc} \frac{x}{y} & - \frac{x^2 + y^2}{y} \\ \frac{1}{y} & -\frac{x}{y} \end{array} \right), \qquad 
		q(z,w) := \left( \begin{array}{cc} \bar{w}  &- \bar{z}\bar{w} \\- \bar{z}\bar{w} & \bar{z}^2 \bar{w} \end{array}\right)
\end{align}		
\begin{enumerate}
	\item $q$ is $\SL(2,\mathbb{R})$-equivariant in the sense that
					$$q(\Psi(z,w))(\cdot,\cdot) = q(z,w)(\Psi^{-1} \cdot, \Psi^{-1}\cdot) \in Q(\Psi j(z)\Psi^{-1})$$ 
	for every $\Psi \in \text{SL}(2,\mathbb{R})$ and $(z,w) \in T^*\mathbb{H} \cong \H\times\C$.
				
	\item For every $z \in \mathbb{H}$ the fibre map $q(z,\cdot) :T^*_z \mathbb{H}  \rightarrow Q(j(z))$ is a complex antilinear isometry.
\end{enumerate}
\end{Lemma}

\begin{proof}
We leave it to the reader to check that the map $(j,q)$ is indeed constructed by combining (\ref{HJ.QeqT}), (\ref{HJ.Qdual}), and (\ref{Hcpx}). Since all these maps are antiholomorphic $\SL(2,\mathbb{R})$-equivariant isometries, it follows immediately that $q$ is also $\SL(2,\mathbb{R})$-equivariant and restricts to complex antilinear isometries along the fibres.
\end{proof}

\subsubsection{Duality}
In our discussion so far, we viewed a covector $J^* \in T^*_J \cJ(\R^2)$ as $\R$-linear map $J^*: T_J \cJ(\R^2) \rightarrow \R$. This extends uniquely to a complex linear map $T_J \cJ(\R^2) \rightarrow \C$ and thus gives rise to the complex linear duality pairing 
	\begin{align} \label{HJ.Qdual0}
				T^*_J \cJ(\R^2) \times T_J \cJ(\R^2) \rightarrow \C,\qquad \langle J^*, \hat{J}\rangle = J^*(\hat{J}) - \textbf{i} J^*(J\hat{J})
	\end{align}
When identifying $T^*_J \cJ(\R^2)$ with $Q(J)$ using (\ref{HJ.QeqT}) and (\ref{HJ.Qdual}), this pairing is given by
\begin{align} \label{HJ.Qdual1}
		Q(J) \times T_J \cJ(\R^2) \rightarrow \C, \qquad \langle q, \hat{J} \rangle_{Q\times T\mathcal{J}} = \frac{q(\hat{J}v,v)}{h_J(v,v)}
\end{align}
where the right hand side does not depend on $v \in \mathbb{R}^2\backslash\{0\}$.

\begin{Lemma} \label{HJ.QLemma2}
Define $(j,q)$ by (\ref{HJ.Qq}). Then
	\begin{align} \label{HJ.Qdual2} w\hat{z} = \overline{\langle q(z,w), dj(z)\hat{z} \rangle_{Q\times T\mathcal{J}}} \end{align}
for all $z \in \mathbb{H}$ and $\hat{z}, w \in \C$. Here we think of $\hat{z} \in T_z \H$, $w \in T_z^*\H$, and define the right hand side by (\ref{HJ.Qdual1}).
\end{Lemma}

\begin{proof}
This follows directly from the construction of the duality pairing. Alternatively, one may use Lemma \ref{HJ.QLemma1} to verify the formula at $z = \textbf{i}$ and then use $\SL(2,\R)$-equivariance of both sides in (\ref{HJ.Qdual2}) to complete the proof.
\end{proof}

\section{Construction of the moduli space \texorpdfstring{$\cM$}{cM}}
Denote by $ X \subset T^*\H$ the unit disc bundle in the cotangent bundle of the hyperbolic plane. Theorem \ref{ThmHK} provides an explicit $S^1 \times \SL(2,\R)$-invariant hyperkähler metric on $X$, which is compatible with the canonical holomorphic symplectic structure on $T^*\H$ and restricts to the hyperbolic metric along $\H$. The existence of such a metric near the zero section follows from general results of Feix \cite{Feix:2001} and Kaledin \cite{Kaledin:1999}, and the main point of this theorem is to provide an explicit formula. We then calculate moment maps for the action of $S^1$ and $\SL(2,\R)$ for the various symplectic forms.

Next, let $(\Sigma,\rho)$ be a closed $2$-dimensional manifold equipped with an area form $\rho$ and denote by $P \rightarrow \Sigma$ its $\SL(2,\R)$-frame bundle. Using Lemma \ref{HJ.QLemma1}, one can identify the space of sections $\cS(P,X)$ of the associated bundle $P(X) := P \times_{\textrm{SL}(2,\R)} X$ with
	$$\cQ_1(\Sigma) := \left\{ (J,\sigma) \, |\, J \in \cJ(\Sigma),\, \sigma \in Q(J),\, |\sigma|_J < 1\right\}$$
where $Q(J)$ is the space of quadratic differentials $\sigma \in \Omega^0(\Sigma, S^2(T^*\Sigma \otimes_J \C))$. The hyperkähler structure on $X$ induces a hyperkähler structure on the bundle $\cQ_1(\Sigma)$ and the general moment map in Theorem \ref{fib.ThmM} can be used to calculate a hyperkähler moment map for the action of $\textrm{Ham}(\Sigma,\rho)$ on $\cQ_1(\Sigma)$. This is the content of Theorem \ref{hk.ThmM}. We show moreover that two of these moment maps extend to moment maps for the action of $\textrm{Symp}_0(\Sigma,\rho)$.

These calculations are the key ingredient in constructing the moduli space $\cM$, which after suitable rescaling and applying standard Moser isotopy arguments, can be described by
	\begin{align}
			\cM := \left\{ (g,\sigma) \in \textrm{Met}(\Sigma)\times Q(J_g)\,\left|\, \begin{array}{c}\bar{\partial}_{J_g} \sigma = 0,\, |\sigma|_g < 1,\\  K_g - \frac{c}{2}|\sigma|_g^2 = \frac{c}{2}\end{array}\right.\right\}\bigg/ \text{Diff}_0(\Sigma)
	\end{align}
where $c = 2\pi(2 - 2\textrm{genus}(\Sigma))/\textrm{vol}(\Sigma,\rho)$ and $J_g \in \cJ(\Sigma)$ is the unique complex structure compatible with $g$. The main feature of this moment map construction is that it induces a natural hyperkähler structure on $\cM$. We will investigate this structure more closely in the final chapter of this article.

\subsection{Hyperkähler extension of the hyperbolic plane}
The unit disc bundle $X \subset T^*\H$ can be viewed as the set
			$$X = \left\{(z,w) \in \mathbb{H}\times\mathbb{C}\,\left| |w| < \frac{1}{\text{Im}(z)}\right.\right\}$$
and the natural $\SL(2,\R)$-action on $T^*\H$ restricts to an action on $X$ given by
	\begin{align} \label{actionSL2X}
		\SL(2,\R)\times X \rightarrow X,\qquad \left(\begin{array}{cc} a & b \\ c & d \end{array}\right) (z,w) = \left( \frac{az + b}{cz + d}, \, (cz + d)^2 w \right).
	\end{align}
Moreover, consider the $S^1$-action on $X$ which is given by rotation of the fibres
	\begin{align} \label{actionS1X}
		S^1\times X \rightarrow X,\qquad \left(e^{\textbf{i}t}, (z,w) \right) \mapsto \left( z, \, e^{\textbf{i}t} w \right).
	\end{align}
As cotangent bundle of a complex manifold, $X$ carries a canonical holomorphic symplectic structure. This consists of the complex structure $J_1 (\hat{z},\hat{w}) := (\textbf{i}\hat{z} , \textbf{i}\hat{w})$ and the complex nondegenerate closed $2$-form $dz \wedge dw \in \Omega_{J_1}^{2,0}(X)$. Denote the real and imaginary part of this form by
		$$\omega_2 := dx \wedge du - dy \wedge dv, \qquad \omega_3 := dx \wedge dv + dy \wedge du$$
where $z = x + \textbf{i}y$ and $w = u + \textbf{i}v$. A Riemannian metric $g \in \textrm{Met}(X)$ is said to be a hyperkähler metric compatible with this holomorphic symplectic structure, if the relations
			$$\omega_i(\cdot, J_i\cdot) = g(\cdot, \cdot)\qquad \text{for $i=1,2,3$}$$
defines a Kähler form $\omega_1$ and integrable complex structures $J_2, J_3$ satisfy the quaternionic relations together with $J_1$. 

\begin{Remark}
By a Lemma of Hitchin (\cite{Hitchin:1987}, Lemma 6.8) it suffices to require that $J_1$, $J_2$ and $J_3$ are almost complex structure. Integrability is then an automatic consequence of the algebraic relations and the closure of the symplectic forms.
\end{Remark}

Feix \cite{Feix:2001} and Kaledin \cite{Kaledin:1999} showed that for any real analytic Kähler manifold, there exists a unique $S^1$-invariant hyperkähler metric defined on some neighbourhood of the zero section in the total space of the cotangent bundle, which is compatible with the canonical holomorphic symplectic structure and extends the Kähler structure of the zero section. 

\begin{Theorem}[\textbf{Donaldson} \cite{Donaldson:2000}] \label{ThmHK}
Define the Riemannian metric $g$ on $X$ by
			$$g = \frac{d\bar{z} dz}{2\text{Im}(z)^2\sqrt{1-r^2}} + \frac{\text{Im}(z)^2}{2\sqrt{1-r^2}} d\bar{w}dw  + \frac{\textbf{i}\text{Im}(z)\bar{w}}{2\sqrt{1-r^2}} d\bar{z} dw - \frac{\textbf{i}\text{Im}(z)w}{ 2\sqrt{1-r^2}} d\bar{w} dz$$
where $r := |w|\text{Im}(z)$. Then $g$ is a $\SL(2,\R)\times S^1$-invariant hyperkähler metric on $X$. It is compatible with the holomorphic symplectic structure and restricts to the hyperbolic metric along $\mathbb{H}\times\{0\}$ with curvature $-1$.
\end{Theorem}

\begin{proof}
A derivation of this formula is given by Donaldson in \cite{Donaldson:2000}, Lemma 16. For a detailed exposition, see \cite{ST:thesis}, Theorem 4.5.1. We verify in the following only that $g$ defines indeed a hyperkähler metric. 

The induced metric on the anti-canonical line bundle $\Lambda^{2,0} T^*\H$ is given by $\det(g_{\bar{\alpha}\beta}) \equiv \frac{1}{4}$. The Levi-Civita connection of $g$ induces on the anti-canonical bundle the unique connection, which is compatible with this metric and the holomorphic structure. In particular, it follows that the induced connection on $\Lambda^{2,0} T^*\H$ is trivial and therefore
				$$\nabla (dz \wedge dw) = 0.$$
Hence $\omega_2 = \textrm{Re}(dz \wedge dw)$ and $\omega_3 = \textrm{Im}(dz \wedge dw)$ are parallel and therefore $(X,J_2, \omega_2)$ and $(X,J_3, \omega_3)$ are Kähler manifolds. Moreover, since $J_1$ is clearly integrable, we also have that $(X,J_1, \omega_1)$ is Kähler. 

It remains to verify that the complex structures $J_1,J_2,J_3$ satisfy the algebraic relations of the quaternions. Since $dz\wedge dw \in \Omega^{2,0}_{J_1}(X)$, it follows $\omega_2(J_1\cdot, \cdot) = \omega_2(\cdot, J_1\cdot)$ and then 
		$$\omega_2(\cdot, J_2 J_1\cdot) = g(\cdot, J_1 \cdot) = - g(J_1 \cdot, \cdot) = - \omega_2(J_1 \cdot, J_2 \cdot) = - \omega_2(\cdot, J_1 J_2 \cdot)$$
This proves implies $J_2 J_1 = - J_1 J_2$. Using again $\omega_2 + \textbf{i}\omega_3 \in \Omega^{2,0}_{J_1}(X)$, it follows $\omega_2(\cdot, \cdot) = \omega_3(J_1 \cdot,\cdot)$ and thus
		$$g(J_2 \cdot, \cdot) = \omega_2(\cdot, \cdot) = \omega_3(J_1 \cdot,\cdot) = g(J_3 J_1 \cdot, \cdot).$$
This proves $J_2 = J_3 J_1$ and then also $J_3 = - J_2 J_1 = J_1 J_2$.
\end{proof}

\begin{Remark}\label{Rmk:incomplete}
The hyperkähler metric on $X$ is not complete.
\end{Remark}

\begin{Remark} \label{Rmk:XHH}
Hodge \cite{Hodge:PhD} showed that there exists as $\SL(2,\R)$-equivariant diffeomorphism $\alpha: X \rightarrow \H \times \overline{\H}$ which identifies the second complex structure $J_2$ on $X$ with $(\textbf{i}, - \textbf{i})$ on $\H \times \overline{\H}$. To understand this construction consider the diagonal embedding of the zero section $Z := \H \times \{0\}$
		$$Z \rightarrow  \H\times\bar{\H}, \qquad (z,0) \mapsto (z,z).$$
Since $Z$ is a totally real submanifold of $X$ with totally real image, it follows from the implicit function theorem that it extends uniquely to a biholomorphic map on suitable open neighbourhoods. Hodge showed by an explicit calculations that such an extension does in fact exist globally and it is given by the formula 
\begin{align} 
	\alpha(z , w ) &= \left( \exp_z\left( \textbf{i}f_z(w)\right), \, \exp_z\left( -\textbf{i}f_z(w)\right) \right) 
\end{align}
where $f_z : T_z^* \H \rightarrow T_z\H$ is given by 
	$$f_z(w) := \textrm{arctanh}\left(-\textrm{Im}(z)|w| \right) \frac{\textrm{Im}(z)^2w}{\textrm{Im}(z)|w|}.$$
For $z = x + \textbf{i}y$, $w = u+ \textbf{i}v$ and $\gamma := \sqrt{1 - y^2(u^2 + v^2)}$ it holds
\begin{align} \label{iso:XHH}
	\alpha(x + \textbf{i}y, u + \textbf{i}v) &= \left(x - \frac{y^2 v}{1- yu} + \textbf{i} \frac{y\gamma}{1 - yu},\, x + \frac{y^2 v}{1 + yu} + \textbf{i} \frac{y\gamma}{1 + yu} \right)		
\end{align}	
By Remark \ref{Rmk:incomplete}, $\alpha$ is not an isometry for the product metric.
\end{Remark}

\subsection{Moment maps on the fibre}

The $S^1$-action on $X$ defined by (\ref{actionS1X}) is Hamiltonian for $\omega_1$ and rotates the symplectic forms $\omega_2$ and $\omega_3$. The moment map for this action with respect to $\omega_1$ yields a Kähler potential for the hyperkähler metric with respect to the second and third complex structure. This is a general feature for hyperkähler manifolds equipped with such an $S^1$-action, which has been observed in \cite{Hitchin:1987b}. We recall the argument in the next Lemma. 

\begin{Lemma}[\textbf{Rotation of the fibres}] \label{Lemma:fibreS1}
Equip $X$ with the hyperkähler structure obtained in Theorem \ref{ThmHK} and consider the $S^1$-action on $X$ defined by (\ref{actionS1X})
\begin{enumerate}
	\item This action is Hamiltonian with respect to $\omega_1$ and generated by
\begin{align} H: X \rightarrow \mathbb{R},\qquad H(z,w) := \sqrt{1 - \text{Im}(z)^2|w|^2}. \end{align}
	\item $H$ is a Kähler potential for the hyperkähler structure with respect to the second and third complex structure, i.e.
		\begin{align}\label{eq:S1Potential} 2 \textbf{i} \bar{\partial}_{J_2} \partial_{J_2} H = \omega_2, \qquad 2 \textbf{i} \bar{\partial}_{J_3} \partial_{J_3} H = \omega_3 \end{align}
\end{enumerate}
\end{Lemma}

\begin{proof}
For $(z,w) \in X$ write $z = x + \textbf{i}y$ and $w= u + \textbf{i}v$. Then 
		\begin{align*}
				\omega_1((0,\textbf{i}w), (\hat{z},\hat{w})) 
								&= - g( (0,w), (\hat{z},\hat{w}) ) 
								 = - 2 \text{Re} \left( \bar{w} g_{\bar{w}z} \hat{z} + \bar{w} g_{\bar{w}w} \hat{w} \right) \\
								&= -\frac{|w|^2 y\hat{y}  + (u\hat{u} + v \hat{v})y^2}{\sqrt{1 - |w|^2y^2}} 
								 = dH(z,w)[\hat{z},\hat{w}].
		\end{align*}
This shows that $v_H(z,w) = (0,\textbf{i}w)$ is the Hamiltonian vector field generated by $H$.

For the second part, denote by $\phi_t(z,w) := (z, e^{\textbf{i}t}w)$ the rotation by $e^{\textbf{i}t}$. Then
$$\cL_{v_H}(\omega_2 + \textbf{i}\omega_3) 
	= \left.\frac{d}{dt}\right|_{t=0} \phi_t^*(\omega_2 + \textbf{i}\omega_3) 
	= \left.\frac{d}{dt}\right|_{t=0} e^{\textbf{i}t} dz \wedge dw 
	= \textbf{i}\omega_2 - \omega_3$$
and therefore $\cL_{v_H} \omega_2 = - \omega_3$ and $\cL_{v_H} \omega_3 = \omega_2$. The identity
	$$dH(J_2 u) = \omega_1(v_H, J_2 u) = g(J_1 v_H, J_2 u) = g(J_3 v_H, u) = \omega_3(v_H, u)$$
then yields
	$$2 \textbf{i} \bar{\partial}_{J_2} \partial_{J_2} H = d (dH\circ J_2) = d \iota(v_H) \omega_3 = \cL_{v_H} \omega_3 = \omega_2.$$
This proves the first equation in (\ref{eq:S1Potential}). The second follows by a similar calculation and this proves the lemma.
\end{proof}

\textbf{Hyperkähler moment map on $X$.} The $\SL(2,\R)$-action on $X$ defined by (\ref{actionSL2X}) preserves all three symplectic forms and admits a hyperkähler moment map. We calculate the first moment map in Proposition \ref{propfiberMM}. The second and third moment map follow from a more general calculation in Proposition \ref{PropMcot} below.

\begin{Proposition}\label{propfiberMM}
Let $\omega_1$ be the symplectic form obtained in Theorem \ref{ThmHK} and let $j: \mathbb{H}\rightarrow \mathcal{J}(\mathbb{R}^2)$ be the isomorphism (\ref{Hcpx}). Then $\mu_1 : X \rightarrow \mathfrak{sl}^*(2,\mathbb{R})$ defined by
						$$\langle\mu_1(z,w), \xi \rangle := -\sqrt{1- \text{Im}(z)^2|w|^2})\text{tr}(j(z)\xi),\qquad \textrm{for $\xi \in \mathfrak{sl}(2,\R)$}$$
is an equivariant moment map for the $\SL(2,\mathbb{R})$-action on $X$ with respect  to $\omega_1$.
\end{Proposition}

\begin{proof}
The proof consists of three steps.\\

\textbf{Step 1:} \textit{For $0 < r < 1$ define $X_r := \{ (z, w) \in X\,|\, |w|\text{Im}(z) = r \}$. Then
		\begin{align} \label{MomentEq1} \omega_1( (\hat{z}_1, \hat{w}_1), (\hat{z}_2, \hat{w}_2)) = \sqrt{1 - r^2} \,\omega_{\mathbb{H}}(\hat{z}_1, \hat{z}_2). \end{align}
for all $(z,w) \in X_r$ and $(\hat{z}_1, \hat{w}_1), (\hat{z}_2, \hat{w}_2) \in T_{(z,w)}X_r$.}\\

It follows from Lemma \ref{Lemma:fibreS1} that $X_r = H^{-1}(\sqrt{1-r^2})$. Hence $X_r/S^1$ is a Marsden--Weinstein quotient and $\omega_1$ induces a well-defined $\SL(2,\R)$-invariant symplectic form on $X_r/S^1$. Since $\SL(2,\R)$ acts transitively on $X_r$, such a form is unique up to scaling and there exists $f(r) \in \mathbb{R}$ such that
		\begin{align*} \omega_1( (\hat{z}_1, \hat{w}_1), (\hat{z}_2, \hat{w}_2)) = f(r) \,\omega_{\mathbb{H}}(\hat{z}_1, \hat{z}_2). \end{align*}
for all $(z,w) \in X_r$ and $(\hat{z}_1, \hat{w}_1), (\hat{z}_2, \hat{w}_2) \in T_{(z,w)}X_r$. We calculate $f(r)$ by evaluating $\omega_1$ at $(\textbf{i}, r) \in X_r$ on the tangent vectors $(1, 0), (\textbf{i}, - r) \in T_{(\textbf{i},r)} X_r$
	$$f(r) = (\omega_1)_{(\textbf{i},r)}((1, 0), (\textbf{i}, - r)) = 2 \text{Im}\left( g_{\bar{z}z}\textbf{i} - g_{\bar{z}w} r \right) = \sqrt{1  - r^2}.$$
This establishes (\ref{MomentEq1}).			\\

\textbf{Step 2:} \textit{For $(z,w) \in X$ with $w \neq 0$, we define radial and angular vector fields by
			$$V_r(z,w) := \left(0,\frac{w}{\text{Im}(z)|w|}\right), \qquad V_{\phi}(z,w) := \left(0,\textbf{i}w\right).$$
For $\xi \in \mathfrak{sl}(2,\R)$ denote by $L_{(z,w)} \xi \in T_{(z,w)}X$ its infinitesimal action. Then
		\begin{align} \label{MomentMapX:eq2.1} \xi j(z) = j(z)\xi \qquad \Longrightarrow \qquad  L_{(z,w)} \xi = -\text{tr}(j(z)\xi) V_{\phi} \end{align}	
		\begin{align} \label{MomentMapX:eq2.2} \xi j(z) = - j(z)\xi \qquad \Longrightarrow \qquad \omega_1(L_{(z,w)} \xi, V_r)  = 0 = \omega_1(L_{(z,w)} \xi, V_{\phi}). \end{align}
}

Assume first that $\xi \in \mathfrak{sl}(2,\R)$ commutes with $j(z)$. By $\SL(2,\mathbb{R})$-invariance of (\ref{MomentMapX:eq2.1}), we may assume $z = \textbf{i}$. Then $\xi$ has the shape
	$$\xi = \left(\begin{array}{cc} 0 & a \\ -a & 0 \end{array}\right)$$
for some $a \in \R$ and (\ref{MomentMapX:eq2.1}) follows from a direct calculation
	$$L_{(\textbf{i},w)} \xi = \left(0 , 2\textbf{i} a w\right) = 2a V_{\phi} = - \text{tr}(J_0\xi_0)V_{\phi}.$$

Assume next that $\xi \in \mathfrak{sl}(2,\R)$ anti-commutes with $j(z)$. Then $\textrm{tr}\left(j(z)\xi \right) = 0$ and in particular $j(z)\xi \in \mathfrak{sl}(2,\R)$. The key observation is the identity
	\begin{align} \label{MomentMapX:eq2.3}
		L_{(z,w)} \left(j(z)\xi \right) = -\textbf{i} L_{(z,w)} \xi \qquad \textrm{for all $(z,w) \in X$.}
	\end{align}
In order to see this, let $J \in \cJ(\R^2)$ and $q \in Q(J)$ and note that
		$$ L_J (J\xi) = [J\xi, J] = J [\xi, J] = J \left(L_{J} \xi\right) $$
		$$ L_q (J\xi) = - q(J\xi\cdot, \cdot) - q(\cdot, J\xi\cdot) = - \textbf{i}q(\xi\cdot, \cdot) - \textbf{i}q(\cdot, \xi\cdot) = \textbf{i} L_q \xi.$$
These equations directly imply (\ref{MomentMapX:eq2.3}), by the antiholomorphic identification $T^*\H \cong \cQ(\R^2)$ in Lemma \ref{HJ.QLemma1}. We can now prove (\ref{MomentMapX:eq2.2}). The first equation $\omega_1(L_{(z,w)} \xi, V_{\phi}) = 0$ follows from Step 1. The second equation follows from this and (\ref{MomentMapX:eq2.3})
\begin{align*}
	\omega_1(L_{(z,w)} \xi_1, V_{r}) 
		&= -\frac{1}{\text{Im}(z)|w|}\omega_1(L_{(z,w)} \xi_1, \textbf{i}V_{\phi}) \\
		&= -\frac{1}{\text{Im}(z)|w|} \omega_1( L_{(z,w)} (j(z)\xi_1), V_{\phi}) = 0.
\end{align*}

\textbf{Step 3}: \textit{$\mu_1$ satisfies the moment map equation
			\begin{align} \label{MomentEq2} \left\langle d \mu_1(z,w)[\hat{z}, \hat{w}] , \xi \right\rangle  = \omega_1 (L_{(z,w)} \xi, (\hat{z},\hat{w})) \end{align}
for every $(z,w) \in X$ and $(\hat{z},\hat{w}) \in T_{(z,w)}X$.}\\

Suppose first $w = 0$. For tangent vectors $(\hat{z},0)$ along the base, the claim follows from Lemma \ref{LemmaMM0}. For tangent vectors $(0,\hat{w})$ along the fibre, the derivative of $\langle \mu_1, \xi \rangle$ in the direction of $(0,\hat{w})$ vanishes. Since $\omega_1(L_{(z,0)}\xi, (0,\hat{w})) = 0$, it follows that (\ref{MomentEq2}) is satisfied in the case $w = 0$.

Suppose next $r := |w|\text{Im}(z) > 0$ and consider the case where $(\hat{z},\hat{w})$ is tangential to $X_r$. Since $L_{(z,w)} \xi$ is also tangential, it follows from (\ref{MomentEq1}) and Lemma \ref{LemmaMM0}
			\begin{align*}  \left\langle d \mu_1(z,w)[\hat{z}, \hat{w}] , \xi \right\rangle  
								&= - \sqrt{1-r^2}\, \text{tr}(d j(z)[\hat{z}]\xi) 
								 = \sqrt{1 - r^2} \,\omega_{\mathbb{H}}(L_{z} \xi, \hat{z}) \\
								&= \omega_1 (L_{(z,w)}\xi, (\hat{z},\hat{w}))
			\end{align*}

Finally consider the case $r := |w|\text{Im}(z) > 0$ and $(\hat{z},\hat{w}) = V_r(z,w)$. The vector fields $V_r$ and $V_{\phi}$ defined in Step 3 satisfy
		 $$\omega_1(V_r(z,w), V_{\phi}(z,w)) = 2 \text{Im}\left( \frac{\bar{w}}{\text{Im}(z)|w|} g_{\bar{w}w} \textbf{i}w \right) = \frac{r}{\sqrt{1-r^2}} $$
Hence, it follows from Step 2 that
			\begin{align*} \langle d \mu_1(z,w) [V_r], \xi \rangle = \frac{r}{\sqrt{1-r^2}} \text{tr}(j(z)\xi) = \omega_1\left(-\text{tr}(j(z)\xi) V_{\phi}, V_r\right) = \omega_1( L_{(z,w)}\xi, V_r).\end{align*}
This completes the proof of the moment map equation (\ref{MomentEq2}).
\end{proof}

\bigbreak
\begin{Proposition} \label{PropMcot}
Let $G$ be a Lie group acting on a smooth complex manifold $Y$. Denote by $\pi: T^*Y \rightarrow Y$ the canonical projection and recall that the tautological $1$-form $\lambda \in \Omega^1(T^*Y, \mathbb{C})$ is defined by
				$$\lambda_{(y,\alpha)}:= \alpha\circ d\pi(y,\alpha):  T_{(y,\alpha)}(T^*Y) \rightarrow \mathbb{C}.$$
The holomorphic symplectic form on $T^*Y$ is then given by
			$$\omega_2 + \textbf{i}\omega_3 = -d\lambda \in \Omega^2(T^*Y,\C).$$
The $G$-action on $Y$ induces a natural action on $T^*Y$. This action is Hamiltonian with respect to $\omega_2$ and $\omega_3$ and admits the moment maps
						$$\langle \mu_2(y, \alpha), \xi \rangle + \textbf{i}\langle\mu_3(y,\alpha), \xi \rangle := \lambda_{(y,\alpha)}(L_{(y,\alpha)}\xi) = \alpha (L_y \xi),\qquad \textrm{for $\xi \in \fg$}.$$
Here $L_y : \mathfrak{g} \rightarrow T_y Y$ and $L_{(y,\alpha)} : \mathfrak{g} \rightarrow T_{(y,\alpha)} T^*Y$ denote the infinitesimal action on $Y$ and $T^*Y$ respectively.
\end{Proposition}

\begin{proof}
Let $g \in G$, $(y,\alpha) \in T^*Y$ and denote by $m_g : T_y Y \rightarrow T_{gy} Y$ the derivative of the action by $g$. Then $g(y,\alpha) = (gy, \alpha \circ m_g^{-1})$ and $g^* \lambda  = \lambda$. Hence the Lie derivative of $\lambda$ in the direction $v_{\xi}(y,\alpha) := L_{(y,\alpha)} \xi$ vanishes. Then, by Cartan's formula, we get
		$$0 = \mathcal{L}_{v_{\xi}} \lambda = d \iota(v_{\xi}) \lambda + \iota(v_{\xi}) d\lambda.$$
This yields $\omega_2(v_{\xi}, \cdot) + \textbf{i}\omega_3(v_{\xi}, \cdot) = d \lambda(v_{\xi})$ and proves the moment map equation.
\end{proof}

\textbf{Hyperkähler Strukture on $\cQ_1(\R^2)$.} Recall from Lemma \ref{HJ.QLemma1} that there exists a canoncial antiholomorphic diffeomorphism $(j, q): X \rightarrow \cQ_1(\R^2)$. Consider on $\cQ_1(\R^2)$ the natural holomorphic symplectic structure, which is obtained by viewing it as subset of $T^*\cJ(\R^2)$. The push forward of the hyperkähler metric on $X$ then provides a compatible hyperkähler metric on $\cQ_1(\R^2)$. The map $(j, q): X \rightarrow \cQ_1(\R^2)$ is antiholomorphic for the first and third complex structure and holomorphic for the second complex structure. The results of this section then translate into the following:

\begin{Proposition} \label{propfiberMQ}
The $S^1$-action on $\cQ_1(\R^2)$, which is defined by rotation of the fibres, is generated by the Hamiltonian
\begin{align} H:\cQ_1(\R^2) \rightarrow \R, \qquad H(J,q) := \sqrt{1 - |q|_J^2}.\end{align}
Moreover, the $\textrm{SL}(2,\R)$-action on $\cQ_1(\R^2)$ admits the hyperkähler moment
\begin{align}
 \langle \mu_1(J,q), \xi \rangle &:= \sqrt{1 - |q|_J^2} \textrm{tr}(J\xi) \\
	\langle \mu_2(J,q) + \textbf{i}\mu_3(J,q), \xi \rangle &:= \langle q, [\xi,J]\rangle_{\cQ\times T\cJ} := \frac{q([\xi,J]v, v)}{h_J(v,v)}
\end{align}					
for $\xi \in \mathfrak{sl}(2,\R)$, where the right hand side does not depend on $v \in \R^2\backslash\{0\}$.
\end{Proposition}
				
\begin{proof}
The first statement follows from Lemma \ref{Lemma:fibreS1}, where we used that $(j,q)$ is antisymplectic with respect to the first symplectic form.
The hyperkähler moment map follows from Proposition \ref{propfiberMM} and Proposition \ref{PropMcot} (see Lemma \ref{HJ.QLemma2} for a discussion of the duality pairing).
\end{proof}

\subsection{Moment maps on the space of sections}

The main result of this subsection is Theorem \ref{hk.ThmM}, which calculates a hyperkähler moment map for the action of $\textrm{Ham}(\Sigma,\rho)$ on $\cQ_1(\Sigma)$ and two moment maps for the action of $\textrm{Symp}_0(M,\rho)$. We begin our discussion with a careful look at the isomorphism $\cS(P,X) \cong \cQ_1(\Sigma)$.

\subsubsection{Geometric description of the sections}
Let $P \rightarrow (\Sigma,\rho)$ be the $\SL(2,\R)$-frame bundle, let $\cQ_1(\R^2)$ be the space of pairs $(j,q)$ with $j \in \cJ(\R^2)$, $q \in Q(j)$ and $|q|_j < 1$. The associated fibration $P(\cQ_1(\R^2)) := P \times_{\SL(2,\R)} \cQ_1(\R^2)$ naturally embeds into $\textrm{End}(T\Sigma) \times S^2(T^*\Sigma\otimes \C)$ via
\begin{equation}\label{embPXQ}
		[(z,\theta), (j, q) ] \mapsto \left( \theta j \theta^{-1}, \theta^*q \right)
\end{equation}
where $z \in \Sigma$, $\theta: \R^2 \rightarrow T_z\Sigma$ is a volume preserving frame, and $(j,q) \in \cQ_1(\R^2)$. On the space of section this yields the identification
	$$\cS(P,\cQ_1(\R^2)) \cong \cQ_1(\Sigma) := \{ (J,\sigma) \,|\, J \in \cJ(\Sigma), \, \sigma \in Q(J),\, |\sigma|_J < 1 \}$$
where $Q(J) = \Omega^1(\Sigma, S^2(T^*\Sigma \otimes_J \C))$ denotes the space of quadratic $J$-differentials.

\begin{Lemma}\label{LemmaJQs}
\hspace{2em} 
	\begin{enumerate}
		\item Any torsion-free $\SL(2,\R)$-connection on $T\Sigma$ induces connections on $P(\cQ_1(\R^2))$ and $\textrm{End}(T\Sigma)\times S^2(T^*\Sigma\otimes \C)$ which are compatible with respect to (\ref{embPXQ}).

		\item The inclusion (\ref{embPXQ}) is $\textrm{Symp}(\Sigma,\rho)$-equivariant.
		
		\item The first symplectic form satisfies the pointwise identity
		$$\left(\omega_1\right)_{(J,q)}((0,\hat{\sigma}_1), (0,\hat{\sigma}_2)) = \frac{\omega_Q(\hat{\sigma}_1, \hat{\sigma}_2)}{\sqrt{1-|\sigma|^2}}$$ 
		for $(J,\sigma) \in \cQ_1(\Sigma)$ and $\hat{\sigma}_i \in Q(J)$. Here we denote by $\omega_Q$ the pointwise symplectic structure on $S^2(T^*\Sigma\otimes_J \C)$ determined by $J$ and $\rho$.			
	\end{enumerate}
\end{Lemma}

\begin{proof}
The first two claims are a matter of unraveling the definitions and left to the reader. The last claim follows from Theorem \ref{ThmHK}.
\end{proof}

\subsubsection{Calculation of the hyperkähler moment map}
The symplectic forms on $P(\cQ_1(\R^2))$ integrate to symplectic forms on $\cQ_1(\Sigma)$
		$$\underline{\omega}_i((\hat{J}_1,\hat{\sigma}_1),(\hat{J}_2,\hat{\sigma}_2)) := \int_{\Sigma} \omega_i((\hat{J}_1,\hat{\sigma}_1),(\hat{J}_2,\hat{\sigma}_2)) \rho.$$
The next theorem calculates moment maps for these symplectic forms and is due to Donaldson (Proposition 17 in \cite{Donaldson:2000}). We present an alternative proof for the second and third part of the theorem, since we found it hard to transform the original argument into a rigorous proof. Also note that the first moment map in \cite{Donaldson:2000} looks slightly different due to alternative conventions.

\begin{Theorem} \label{hk.ThmM}
The action of $\textrm{Ham}(\Sigma,\rho)$ on $\cQ_1(\Sigma)$ is Hamiltonian for all three symplectic structures $\underline{\omega}_i$. Moreover, the action of $\textrm{Symp}_0(\Sigma,\rho)$ on $\cQ_1(\Sigma)$ is Hamiltonian for $\underline{\omega}_2$ and $\underline{\omega}_3$.
\begin{enumerate}

\item An equivariant moment map for the $\text{Ham}(\Sigma,\rho)$-action on $\cQ_1(\Sigma)$ for $\underline{\omega}_1$ is
			\begin{align} \label{hk.Meq1}
							\underline{\mu}_1 (J,\sigma) = \frac{|\partial_J \sigma|_J^2 - |\bar{\partial}_J \sigma|_J^2 }{\sqrt{1 - |\sigma|_J^2}} \rho - 2 \sqrt{1 - |\sigma|_J^2} K_J \rho - 2 \textbf{i} \bar{\partial}_J \partial_J \sqrt{1 - |\sigma|_J^2}  + 2c\rho
			\end{align}
		where $c := 2\pi(2 - 2\textrm{genus}(\Sigma))/\textrm{vol}(\Sigma,\rho)$, $K_J$ denotes the Gaussian curvature of $\rho(\cdot,J\cdot)$ and all norms $|\cdot|_J$ are calculated with respect to this metric.

\item Define the contraction $r: \Omega^{0,1}_J(\Sigma, S^2(T^*\Sigma \otimes_J \C)) \rightarrow \Omega^{1,0}_J(\Sigma)$ by 
			$r(\gamma) :=  \frac{\gamma(v)(v, \cdot)}{|v|_J^2}$
which is independent of $0\neq v \in \textrm{Vect}(\Sigma)$. An equivariant moment map for the $\text{Ham}(\Sigma,\rho)$-action on $\cQ_1(\Sigma)$ for $\underline{\omega}_2$ and $\underline{\omega}_3$ is given by
\begin{align}\label{hk.Meq3}
\underline{\mu}_2(J,\sigma) + \textbf{i}\underline{\mu}_3(J,\sigma) = 2\textbf{i} \bar{\partial}_J r(\bar{\partial}_J \sigma)
\end{align}
			
\item An equivariant moment map for the $\textrm{Symp}_0(\Sigma,\rho)$-action on $\cQ_1(\Sigma)$ with respect to $\underline{\omega}_2$ and $\underline{\omega}_3$ is given by
\begin{align} \label{hk.Meq2}
	\left\langle \underline{\tilde{\mu}}_2(J,\sigma) + \textbf{i}\underline{\tilde{\mu}}_3(J,\sigma), v \right\rangle 
	= 2\textbf{i}\int_{\Sigma} \iota(v) r(\bar{\partial}_J \sigma) \rho.
\end{align}			
for any symplectic vector field $v \in \textrm{Vect}(\Sigma)$ satisfying $d\iota(v) \rho = 0$.
\end{enumerate}

\end{Theorem}

\begin{proof}
Denote by $\nabla$ the Levi-Civita connection of $\rho(\cdot,J\cdot)$. We deduce (\ref{hk.Meq1}) in Step 1  from Theorem \ref{fib.ThmM}. For the proof of (\ref{hk.Meq2}) we need to extend the arguments used in the derivation of Theorem \ref{fib.ThmM}. This is done in Step 2 and the derivation of (\ref{hk.Meq3}) and (\ref{hk.Meq2}) is completed in Step 3 and Step 4.\\

\textbf{Step 1:} \textit{(\ref{hk.Meq1}) defines an equivariant moment map for the action of $\textrm{Ham}(\Sigma,\rho)$ with respect to $\underline{\omega}_1$.}\\

Since $\nabla J = 0$ for the Levi-Civita connection, Lemma \ref{LemmaJQs} shows
	\begin{align} \label{hkM:proof11a}
		\underline{\omega}_1(\nabla (J,\sigma) \wedge \nabla (J,\sigma)) = \underline{\omega}_1((0,\nabla \sigma) \wedge (0, \nabla \sigma)) =  \frac{\omega_Q( \nabla \sigma, \nabla \sigma)}{\sqrt{1 - |\sigma|^2}}.
	\end{align}
Moreover, for $u \in \textrm{Vect}(\Sigma)$, we calculate
	\begin{align*}
		|\partial_u \sigma|_J^2 - |\bar{\partial}_u \sigma|_J^2 
			= \frac{1}{4} \left| \nabla_{u} \sigma - \textbf{i} \nabla_{Ju} \sigma\right|_J^2  - \frac{1}{4} \left| \nabla_{u} \sigma + \textbf{i} \nabla_{Ju} \sigma\right|_J^2  
			= \omega_Q( \nabla_u \sigma, \nabla_{Ju} \sigma).
		\end{align*} 
Hence $\left(|\partial \sigma|^2 - |\bar{\partial}\sigma|^2\right) \rho = \omega_Q( \nabla \sigma, \nabla \sigma)$ and with (\ref{hkM:proof11a}) it follows
\begin{align}\label{hkM.proof.11}
	\underline{\omega}_1(\nabla (J,\sigma), \nabla (J,\sigma)) = \frac{|\partial_J \sigma|_J^2 - |\bar{\partial}_J \sigma|_J^2}{\sqrt{1 - |\sigma|_J^2}} \rho
\end{align}

The Riemann curvature tensor $R^{\nabla}$ and the Gaussian curvature $K_J$ are related by the formula $R^{\nabla} = - K_J J\otimes \rho$. Using Proposition \ref{propfiberMQ} it follows
\begin{align} 
	\label{hkM.proof.12} \langle \mu_{(J,\sigma)}, R^{\nabla} \rangle = -\sqrt{1 - |\sigma|_J^2} K_J \text{tr}(J^2) \rho = 2K_J\sqrt{1- |\sigma|_J^2} \rho.
\end{align}

Finally, using Proposition \ref{propfiberMM} we obtain
	\begin{align*}
		\nabla_u \mu_{(J,\sigma)} (\Psi)
		&= \cL_u \left( \sqrt{1 - |\sigma|_J^2} \textrm{tr}(J\Psi) \right) - \sqrt{1 - |\sigma|_J^2}) \textrm{tr}(J\nabla_u\Psi) \\
		&= \cL_u \left( \sqrt{1 - |\sigma|_J^2} \right) \textrm{tr}(J\Psi)
	\end{align*}
for all $\Psi \in \Omega^0(\Sigma, \textrm{End}(T\Sigma))$ and $u \in \textrm{Vect}(\Sigma)$. Let $e_1$, $e_2 = J e_1$ be a local orthonormal frame for $T\Sigma$ and write $v \in \textrm{Vect}(\Sigma)$ as $v= v_1 e_1 + v_2 e_2$. Then
		\begin{align*}
				c (\nabla \mu_{(J,\sigma)}) (v) &= \nabla_{e_1} \mu_{(J,\sigma)}(e_1^*\otimes v) + \nabla_{e_2} \mu_{(J,\sigma)}(e_2^* \otimes v) \\
														 &= \cL_{e_1} \sqrt{1- |\sigma|_J^2} \textrm{tr}(J (e_1^*\otimes v)) + \cL_{e_2} \sqrt{1- |\sigma|_J^2} \textrm{tr}(J (e_2^*\otimes v)) \\
														 &= - \cL_{e_1} \sqrt{1- |\sigma|_J^2} v_2 + \cL_{e_2} \sqrt{1- |\sigma|_J^2} v_1\\
														 &= \cL_{Jv} \left(\sqrt{1- |\sigma|_J^2}\right).
		\end{align*}
Using the relation $d (df \circ J) = 2 \textbf{i} \bar{\partial} \partial f$ for $f(z) := - \sqrt{1- |\sigma|_J^2}$, it follows
\begin{align} \label{hkM.proof.13}
		d c(\nabla \mu_{(J,\sigma)}) = 2 \textbf{i}\bar{\partial} \partial \sqrt{1- |\sigma|_J^2}.
\end{align}		

We have identified in (\ref{hkM.proof.11}), (\ref{hkM.proof.12}) and (\ref{hkM.proof.13}) the three components of the moment map in Theorem \ref{fib.ThmM}. The cohomology class of
		$$\frac{|\partial_J \sigma|_J^2- |\bar{\partial}_J \sigma|_J^2}{\sqrt{1 - |\sigma|_J^2}} \rho - 2 \sqrt{1 - |\sigma|_J^2} K_J \rho - 2 \textbf{i} \bar{\partial} \partial \sqrt{1 - |\sigma|_J^2}$$
does not depend on $(J,\sigma) \in \cQ_1(\Sigma)$ and by the Gauss--Bonnet theorem it is represented by $-2c\rho$. Therefore $\underline{\mu}_1$ takes values in the space of exact $2$-forms and the moment map equation follows from Theorem \ref{fib.ThmM}.\\

\textbf{Step 2:} \textit{ For $J \in \cJ(\Sigma)$, $\sigma \in Q(J)$ and $\hat{J} \in T_J \cJ(\Sigma)$ consider the duality pairing
	\begin{align} \langle \sigma, J \rangle_{Q\times T\cJ} : \Sigma \rightarrow \C, \qquad \langle \sigma, \hat{J} \rangle_{Q\times T\cJ} := \frac{\sigma(\hat{J}u, u)}{|u|_J^2} \end{align}
where the right hand side does not depend on the choice of a locally defined vector field $u \neq 0$. An equivariant moment map for the $\textrm{Symp}_0(\Sigma,\rho)$-action on $\cQ_1(\Sigma)$ with respect to $\underline{\omega}_2$ and $\underline{\omega}_3$ is given by
	\begin{align}\left\langle \left(\underline{\mu}_2 + \textbf{i} \underline{\mu}_3\right) (J,\sigma), v \right\rangle := -2\textbf{i} \int_{\Sigma} \langle \sigma, \bar{\partial}_J v \rangle_{Q\times T\cJ} \rho\end{align}
for any symplectic vector field $v \in \textrm{Vect}(\Sigma)$ satisfying $d\iota(v) \rho = 0$.
}\\

The Lie derivative of $J$ and $\sigma$ along $v$ are given by
	$$ \cL_v (J, \sigma) = (\nabla_v J, \nabla_v \sigma) - L_{(J,\sigma)} (\nabla v)$$
where $L_{(J,\sigma)}$ denotes infinitesimal action of $\textrm{End}_0(T\Sigma)$ along the fibre (see Lemma \ref{fib.LemmaInf1}). For a smooth path $\R \rightarrow \cQ_1(\Sigma)$, $t \mapsto (J_t,\sigma_t)$, consider
\begin{equation} \label{hkM:proof2}
\begin{aligned}	
	\left(\underline{\omega}_2 + \textbf{i} \underline{\omega}_3\right) \left( (\dot{J}_t, \dot{\sigma}_t), \cL_v (J_t, \sigma_t) \right) 
			&= \left(\underline{\omega}_2 + \textbf{i} \underline{\omega}_3\right) \left( (\dot{J}_t, \dot{\sigma}_t), \nabla_v (J_t, \sigma_t) \right) \\
				& \qquad - \left(\underline{\omega}_2 + \textbf{i} \underline{\omega}_3\right) \left( (\dot{J}_t, \dot{\sigma}_t), L_{(J_t,\sigma_t)} (\nabla v)\right) 
\end{aligned}		
\end{equation}
Using the definitions and integration by part, the first term yields:
\begin{align*}
&\left(\underline{\omega}_2 + \textbf{i} \underline{\omega}_3\right) \left( (\dot{J}_t, \dot{\sigma}_t), \nabla_v (J_t, \sigma_t) \right) \\
	&\qquad= \int_{\Sigma} \langle \dot{\sigma}_t, \nabla_v J_t \rangle_{Q\times T\cJ} \,\rho - \langle \nabla_v\sigma_t, \dot{J}_t \rangle_{Q\times T\cJ} \,\rho \\
	&\qquad= \int_{\Sigma} \langle \dot{\sigma}_t, \nabla_v J_t \rangle_{Q\times T\cJ} \wedge \iota(v) \rho - \int_{\Sigma} \langle \nabla \sigma_t, \dot{J}_t \rangle_{Q\times T\cJ} \wedge \iota(v) \rho \\
	&\qquad= \int_{\Sigma} \langle \dot{\sigma}_t, \nabla_v J_t \rangle_{Q\times T\cJ} \wedge \iota(v) \rho + \int_{\Sigma} \langle \sigma_t, \nabla \dot{J}_t \rangle_{Q\times T\cJ} \wedge \iota(v) \rho \\
	&\qquad= \partial_t \int_{\Sigma} \langle \sigma_t, \nabla_v J_t \rangle_{Q\times T\cJ} \wedge \iota(v) \rho
\end{align*}
For the second term, apply Proposition \ref{propfiberMQ} fibrewise to obtain
\begin{align*}
\left(\underline{\omega}_2 + \textbf{i} \underline{\omega}_3\right) \left( (\dot{J}_t, \dot{\sigma}_t), L_{(J_t,\sigma_t)} (\nabla v)\right)
&= \int_{\Sigma} (\omega_2 + \textbf{i} \omega_3)\left( (\dot{J}_t, \dot{\sigma}_t), L_{(J_t,\sigma_t)} (\nabla v)\right) \rho \\
&= -\partial_t \int_{\Sigma} \langle \sigma_t, L_{J_t} (\nabla v) \rangle_{Q\times T\cJ} \, \rho \\
&= 2 \textbf{i} \partial_t \int_{\Sigma} \langle \sigma_t, \bar{\partial}_{J_t} v \rangle_{Q\times T\cJ} \, \rho.
\end{align*}
Here we used in the last equation that $L_{J_t} (\nabla v) = [\nabla v, J_t] = -2 J_t \bar{\partial}_{J_t} v$. It follows from this calculation that a moment map with respect to $\underline{\omega}_2$ and $\underline{\omega}_3$ is given by 
$$ \left\langle \left(\underline{\mu}_2 + \textbf{i} \underline{\mu}_3\right) (J,\sigma), v \right\rangle 
	=  \int_{\Sigma} \langle \sigma, \nabla_v J \rangle_{Q\times T\cJ} \wedge \iota(v) \rho	-2\textbf{i} \int_{\Sigma} \langle \sigma, \bar{\partial}_J v \rangle_{Q\times T\cJ} \rho	$$
The first term vanishes for the Levi-Civita connection, since $\nabla J = 0$, and this proves Step 2.\\

\textbf{Step 3:} \textit{(\ref{hk.Meq2}) defines an equivariant moment maps for the action of $\textrm{Symp}_0(\Sigma,\rho)$ with respect to $\underline{\omega}_2$ and $\underline{\omega}_3$.}\\

Step 3 follows Step 2 and Stokes theorem using the pointwise identity
\begin{align}
	\langle \sigma, \bar{\partial}_J v \rangle_{Q\times T\cJ} \rho = -\frac{1}{2\textbf{i}} \bar{\partial}_J (\iota(v) q) - (\iota(v) r(\bar{\partial}_J \sigma) ) \rho .
\end{align}
This can either be verified in a local holomorphic chart or deduced from the following calculation, which holds for any locally defined and parallel vector field $u \neq 0$:
\begin{align*}
\sigma(\bar{\partial}_u v, u) 
	&= \bar{\partial}_u \left( \sigma(v, u) \right) - (\bar{\partial}_u \sigma)(v, u) \\
	&= -\frac{1}{2\textbf{i}} \left[ \bar{\partial}_u \left(\sigma(v, Ju) \right) - \bar{\partial}_{Ju} \left(\sigma(v, u) \right)\right] - \iota(v)r(\bar{\partial}_u \sigma) |u|_J^2 \\
	&= - \frac{1}{2\textbf{i}} \bar{\partial} (\iota(v) \sigma) (u, Ju) - \iota(v)r(\bar{\partial}_u \sigma) \rho(u,Ju)
\end{align*}

\textbf{Step 4:} \textit{(\ref{hk.Meq3}) defines an equivariant moment maps for the action of $\textrm{Ham}(\Sigma,\rho)$ with respect to $\underline{\omega}_2$ and $\underline{\omega}_3$.}\\

Let $H: \Sigma \rightarrow \R$ be a Hamiltonian and define $v_H \in \textrm{Vect}(M)$ by $\iota(v_H)\rho = dH$. Then
\begin{align*}
	\int_{\Sigma} \iota(v_H) r(\bar{\partial}_J \sigma) \rho 
	= \int_{\Sigma} r(\bar{\partial}_J \sigma) \wedge dH 
	= \int_{\Sigma} H  \bar{\partial}_J r(\bar{\partial}_J \sigma) 
\end{align*}
and (\ref{hk.Meq3}) follows now from Step 3.
\end{proof}

\subsection{Construction of the moduli space}

\subsubsection{The Hamiltonian quotient space}
The hyperkähler quotient of $\cQ_1(\Sigma)$ by $\text{Ham}(\Sigma,\rho)$ is defined by
\begin{equation} \label{moduli:QM0}
\begin{aligned}
		\mathcal{M}_0 &:= \underline{\mu}_1^{-1}(0)\cap\underline{\mu}_2^{-1}(0)\cap\underline{\mu}_3^{-1}(0)/ \text{Ham}(\Sigma,\rho) \\
									&=  \left\{(J,\sigma) \in \cQ_1(\Sigma)\,\left|\, \underline{\mu}_1(J,\sigma) = 0, \, \bar{\partial}_J r(\bar{\partial}_J \sigma) = 0 \right. \right\} \bigg/ \text{Ham}(\Sigma,\rho)
\end{aligned}
\end{equation}
where $\underline{\mu}_1, \underline{\mu}_2, \underline{\mu}_3$ are the moment maps calculated in Theorem \ref{hk.ThmM} for the $\textrm{Ham}(\Sigma,\rho)$-action on $\cQ_1(\Sigma)$. It follows from general principles that $\mathcal{M}_0$ is a hyperkähler manifold. 
The next lemma is formulated in a finite dimensional setting, but extends formally to our case. It indicates that transversality for the hyperkähler moment map is an automatic consequence of the setup.

\begin{Lemma}
Let $(M,g, I_1,I_2,I_3)$ be a hyperkähler manifold and let $G$ be a Lie group. Suppose $G$ acts freely on $M$ by hyperkähler isometries and admits a hyperkähler moment map
						$$\mu = (\mu_1, \mu_2, \mu_3) : M \rightarrow \mathbb{R}^3\otimes \mathfrak{g}^*.$$
Then $0$ is a regular value of $\mu$.
\end{Lemma}

\begin{proof}
Let $x \in M$ with $\mu(x) = 0$ be given. Denote by $L_x: \mathfrak{g} \rightarrow T_xM$ its infinitesimal action and decompose $T_x M = W_0 \oplus W_1$ with
					$$W_1 := \text{Im}(L_x), \qquad W_0 = \text{Im}(L_x)^{\perp}.$$
Equivariance of the moment map yields for all $\xi,\eta \in \mathfrak{g}$ the identity
		$$\langle I_i L_x \xi, L_x \eta \rangle = \langle d\mu_i(x) L_x \eta, \xi \rangle = \langle \mu_i(x), [\eta, \xi] \rangle = 0.$$
This shows that the three complex structures map $W_1$ into $W_0$. 

Let $\eta_1,\eta_2,\eta_3 \in \mathfrak{g}^*$ be given. Since $G$ acts freely, $L_x$ is injective, and the dual map $L_x^*: T_xM \rightarrow \mathfrak{g}^*$ is surjective with kernel $W_0$. Hence there exist $u_i\in W_1$ with $\eta_i = L_x^*(u_i)$ for $i = 1,2,3$. For $v :=  I_1 u_1 + I_2 u_2 + I_3 u_3$ we then obtain
		$$\langle d\mu_i(x)v , \xi \rangle = \omega_i (L_x\xi, v) = - g(L_x\xi, I_i v ) = g(L_x\xi, u_i ) = \langle \eta_i, \xi \rangle.$$
This proves surjectivity of $d\mu : T_xM \rightarrow \mathbb{R}^3\otimes \mathfrak{g}^*$ and the lemma.
\end{proof}

\subsubsection{Construction of the symplectic quotient space}		
The group a Hamiltonian diffeomorphism $\textrm{Ham}(\Sigma,\rho) < \textrm{Symp}_0(\Sigma,\rho)$ is a normal subgroup (see \cite{IntroductionST}, Proposition 10.2) and therefore 
		$$H := \text{Symp}_0(\Sigma,\rho)/\text{Ham}(\Sigma,\rho)$$
is a well-defined quotient group. The flux homomorpism associates to every path $[0,1] \rightarrow \text{Symp}_0(\Sigma,\rho)$, $t \mapsto \psi_t$, a cohomology class in $H^1(\Sigma, \R)$ defined by
						$$\text{Flux}(\{\psi_t\}) := \int_0^1 [\iota(\partial_t \psi_t) \omega ] \, dt \in H^1(\Sigma,\mathbb{R}).$$
Since $\pi_1(\text{Symp}_0(\Sigma)) = 0$ for a closed higher genus surface (see \cite{1969:EarlyEells}), it follows that the flux homomorphism descends to an isomorphism
		$$\textrm{Flux}: \, H := \textrm{Symp}_0(\Sigma,\rho) / \textrm{Ham}(\Sigma,\rho) \stackrel{\cong}{\longrightarrow} H^1(\Sigma,\R).$$
See \cite{IntroductionST}, Proposition 10.18 for more details. The Lie algebra of $H$ is the quotient
	\begin{align*}
		\textrm{Lie}(H) &:= \frac{\{ v \in \textrm{Vect}(\Sigma)\,|\, d\iota(v)\rho = 0 \}}{\{ v \in \textrm{Vect}(\Sigma)\,|\, \textrm{$ \iota(v)\rho$ exact} \}} 
	\end{align*}
The Hamiltonian quotient $\cM_0 := \cQ_1/\!/\textrm{Ham}(\Sigma,\rho)$ defined by (\ref{moduli:QM0}) admits a natural action of $H$ which preserving the hyperkähler structure. We investigate this action more closely in the following proposition.

\begin{Proposition} \label{hk.Q.Prop} \hspace{2em}
		\begin{enumerate}			
			\item The $H$-action is Hamiltonian with respect to $\underline{\omega}_2$ and $\underline{\omega}_3$. For $[v] \in \textrm{Lie}(H)$ and $[J,\sigma] \in \cM_0$ the maps
					\begin{align} \label{hk.Q.eq0}
		\left\langle\tilde{\mu}^H_2([J,\sigma]) + \textbf{i} \tilde{\mu}^H_3([J,\sigma]), [v] \right\rangle = 2\textbf{i} \int_{\Sigma} \iota(v)r(\bar{\partial}_J\sigma) \rho
					\end{align}		
			are well-defined equivariant moment maps for $\underline{\omega}_2$ and $\underline{\omega}_3$ respectively.
			\item On $\mathcal{M}_0$ the equation $\tilde{\mu}_2^H([J,\sigma]) = 0$ is equivalent to $\tilde{\mu}_3^H([J,\sigma]) = 0$ and 
	\begin{align} \label{hk.Q.eq1}
		(\tilde{\mu}_2^H)^{-1}(0) = (\tilde{\mu}_3^H)^{-1}(0) = \left\{ [J,\sigma] \in \mathcal{M}_0\,|\, \bar{\partial}_J \sigma = 0\right\}
	\end{align}			
			
			\item $(\tilde{\mu}_2^H)^{-1}(0) = (\tilde{\mu}_3^H)^{-1}(0)$ is a $J_1$--complex submanifold
			
			\item The $H$-orbits in $(\tilde{\mu}_2^H)^{-1}(0) = (\tilde{\mu}_3^H)^{-1}(0)$ are $J_1$--complex submanifolds.
			
		\end{enumerate}

\end{Proposition}

\begin{proof}
Theorem \ref{hk.ThmM} show that the action of $\textrm{Symp}_0(\Sigma,\rho)$ on $\cQ_1(\Sigma)$ is Hamiltonian with respect to $\underline{\omega}_2$ and $\underline{\omega}_3$. This directly implies that the action of $H$ is Hamiltonian for the symplectic forms induced by $\underline{\omega}_2$ and $\underline{\omega}_3$. The formula for the moment maps (\ref{hk.Q.eq0}) follows from (\ref{hk.Meq2}).

For fixed $J$, one can identify the Lie algebra of $H$ with
		$$\fh_J :=  \{v \in \text{Vect}(\Sigma)\,|\, d \iota(v)\rho = 0 = d \iota(Jv)\rho  \}$$
by Hodge theory. This is a $J$-invariant subspace and it holds $\tilde{\mu}_2^H([J,\sigma], Jv) = \tilde{\mu}_3^H([J,\sigma], v)$ for all $v \in \fh_J$. Hence $(\tilde{\mu}^H_2)^{-1}(0) = (\tilde{\mu}^H_3)^{-1}(0)$.

We prove (\ref{hk.Q.eq1}) next. Let $[J,\sigma] \in (\tilde{\mu}_2^H)^{-1}(0) = (\tilde{\mu}_3^H)^{-1}(0)$ be given. Then
		$$0 = \int_{\Sigma} \iota(v)r(\bar{\partial}_J\sigma) \rho = \int_{\Sigma} r(\bar{\partial}_J\sigma) \wedge \iota(v)\rho$$
for all $v \in \fh_J$. The defining equations of $\mathcal{M}_0$ show that $r(\bar{\partial}_J\sigma) \in \Omega_J^{1,0}(\Sigma,\C)$ is closed. Since $\{\iota(v)\rho\,|\,v\in\mathfrak{h}_J\}$ parametrizes the space of (real) harmonic $1$-forms, it follows from Poincar\'e duality that $[r(\bar{\partial}_J\sigma)] = 0 \in H_J^{1,0}(\Sigma)$. Hence $r(\bar{\partial}_J \sigma) = \partial_J f$ with $\bar{\partial}_J\partial_J f = 0$. Then $f$ is constant and therefore $\bar{\partial}_J \sigma = 0$.

Let $[J,\sigma] \in (\tilde{\mu}_2^H)^{-1}(0) = (\tilde{\mu}_3^H)^{-1}(0)$. It follows from the moment map equations that multiplication with $J_2$ and $J_3$ yields isomorphism
		$$ J_2: T_{[J,\sigma]} \left(H \cdot [J,\sigma]\right) \rightarrow \left(T_{[J,\sigma]} \left(\tilde{\mu}_2^H\right)^{-1}(0) \right)^{\perp}$$
		$$ J_3: T_{[J,\sigma]} \left(H \cdot [J,\sigma]\right) \rightarrow \left(T_{[J,\sigma]} \left(\tilde{\mu}_3^H\right)^{-1}(0) \right)^{\perp}$$
Hence $J_1 = J_2J_3$ maps $T_{[J,\sigma]} \left(H \cdot [J,\sigma]\right)$ and $T_{[J,\sigma]} \left(\tilde{\mu}_2^H\right)^{-1}(0)$ onto themselves. Therefore $\left(\tilde{\mu}_2^H\right)^{-1}(0) = \left(\tilde{\mu}_3^H\right)^{-1}(0)$ is a $J_1$-complex submanifold of $\cM_0$ and the $H$-orbits are complex submanifolds.
\end{proof}

The next lemma describes a general procedure to obtain symplectic quotients in the absence of moment maps.

\begin{Lemma} \label{Teich.LemmaSymplStr}
Let $(Q,\omega)$ be a symplectic manifold and let $G$ be a Lie group acting symplectically, properly and freely on $Q$. Suppose that all $G$ orbits are symplectic submanifold of $Q$. Then $Q/G$ carries a natural symplectic structure which is obtained by declaring that $(T_q (G\cdot q))^{\omega} \rightarrow T_{[q]} Q/G$ is a symplectic isomorphism for every $q \in Q$.
\end{Lemma}

\begin{proof}
The tangent space $T_q (G\cdot q)$ of the $G$-orbit through $q$ is by assumption symplectic and so its symplectic complement $(T_q (G\cdot q))^{\omega}$ is also a symplectc.  The induced symplectic form on $T_{[q]} Q/G$ does not depend on the representative $q$, because $G$ acts symplectically on $Q$. It follows that $Q/G$ carries a well-defined non-degenerated $2$-form $\omega_{Q/G} \in \Omega^2(Q/G)$. It remains to show that $\omega_{Q/G} \in \Omega^2(Q/G) $ is closed. We have
\begin{align*}
				d\omega_{Q/G}(v_1,v_2,v_3) &= \omega_{Q/G}([v_1,v_2],v_3) + \omega_{Q/G}([v_2,v_3],v_1) + \omega_{Q/G}([v_3,v_1],v_2) \\
																	 &\,\, -\mathcal{L}_{v_3}(\omega_{Q/G}(v_1,v_2)) - \mathcal{L}_{v_1}(\omega_{Q/G}(v_2,v_3)) - \mathcal{L}_{v_2}(\omega_{Q/G}(v_3,v_1))
\end{align*}
for $v_1,v_2,v_3 \in \textrm{Vect}(Q/G)$. Let $\tilde{v}_j \in \textrm{Vect}(Q)$ be the unique lift of $v_j$ with $\tilde{v}_j(q) \in
(T_q (G\cdot q))^{\omega}$ for all $q \in Q$. Then follows 
		$$\omega_{Q/G}([v_1,v_2], v_3) = \omega([\tilde{v}_1, \tilde{v}_2], \tilde{v}_3)$$
since $[\tilde{v}_1,\tilde{v}_2]$ projects to $[v_1,v_2]$. Moreover, $\mathcal{L}_{v_3}(\omega_{Q/G}(v_1,v_2)) = \mathcal{L}_{\tilde{v}_3}(\omega(\tilde{v}_1,\tilde{v}_2))$. Similar equations hold for the other terms in $d\tilde{\omega}$ and hence
			$$d\omega_{Q/G}(v_1,v_2,v_3) = d\omega(\tilde{v}_1,\tilde{v}_2,\tilde{v}_3) = 0.$$
This completes the proof of the lemma.
\end{proof}

Consider the moduli space
\begin{align}  \label{moduli:QMa}
	\cM_s	:=  \left\{(J,\sigma) \in \cQ(\Sigma)\,\left|\, \underline{\mu}_1(J,\sigma) = 0, \, \bar{\partial}_J \sigma = 0,\, |\sigma| < 1\right.\right\} \bigg/ \text{Symp}_0(\Sigma,\rho)
\end{align}
It follows from Proposition \ref{hk.Q.Prop} that $\cM_s$ is a Marsden--Weinstein quotient of $\cM_0$ for the symplectic structures induced by $\underline{\omega}_2$ and $\underline{\omega}_3$, and therefore they descend both to symplectic structures on $\cM_s$. Using Lemma \ref{Teich.LemmaSymplStr}, $\underline{\omega}_1$ also induces a natural symplectic structure on $\cM_s$. This yields three algebraically compatible symplectic forms on $\cM_s$ and a lemma of Hitchin (\cite{Hitchin:1987}, Lemma 6.8) shows that this defines a hyperkähler structure.

\begin{Proposition} \label{hk.Prop:SimplifyEq}
Let $(J,\sigma) \in \cQ_1(\Sigma)$ and assume $\bar{\partial}_J \sigma = 0$. Then
				\begin{align} \label{geoJQ}
					\underline{\mu}_1(J,\sigma) = - \left(2K_J + \Delta \log(1 + \sqrt{1- |\sigma|^2})\right)\rho + 2c\rho
				\end{align}
were $\Delta = d^*d$ is the positive Laplacian for the metric $\rho(\cdot,J\cdot)$. In particular, 
\begin{align}  \label{moduli:QMb}
	\mathcal{M}_s = \left\{ (J,\sigma)\in \cQ(\Sigma)\,\left|\, \begin{array}{c} \bar{\partial}_J \sigma = 0, \quad |\sigma| < 1 \\ K_J + \frac{1}{2}\Delta \log(1 + \sqrt{1- |\sigma|^2}) = c \end{array}\right.\right\} \bigg/ \text{Symp}_0(\Sigma, \rho)
\end{align}
where $c := 2\pi(2 - 2\textrm{genus}(\Sigma))/\textrm{vol}(\Sigma,\rho)$.
\end{Proposition}

\begin{proof}
Consider the function $h := |\sigma|^2$. It suffices to prove the lemma around a point where $h \neq 0$. We also simplify notation and abbreviate $\bar{\partial} := \bar{\partial}_J$ and $\partial := \partial_J$.
			
Denote by $h_Q$	the hermitian form on the bundle of quadratic differentials induced by $J$ and $\rho$. Since $\bar{\partial}\sigma = 0$, it follows $\partial h = \partial h_Q(\sigma, \sigma) = h_Q(\sigma, \partial \sigma)$ and hence $|\partial h|^2 = h |\partial \sigma|^2$.
Moreover, using $|\partial h |^2 \rho = -\frac{\textbf{i}}{2}\, \bar{\partial} h \wedge \partial h$, we then obtain
\begin{align} \label{geoJQeq2} |\partial \sigma|^2 \rho = -\frac{\textbf{i}}{2h} \bar{\partial} h \wedge \partial h.\end{align}
Next, choose holomorphic coordinates and write
				$$\rho = \lambda dx \wedge dy, \qquad \sigma(z) = f(z) dz^2$$
for some positive function $\lambda$ and a holomorphic function $f$. The Gaussian curvature $K_J$ can be computed in these coordinates via
				$$K_J = - \frac{1}{2} \lambda^{-1} (\partial_x^2 \log(\lambda) + \partial_y^2 \log(\lambda)).$$
Since $f(z)$ is holomorphic, $\log(|f(z)|^2)$ is harmonic and we compute
	\begin{align*}
		\bar{\partial} \partial \log(h) &= \frac{1}{4} (\partial_x^2 + \partial_y^2)\log(|f(z)|^2 \lambda^{-2}) 2\textbf{i} \,dx \wedge dy\\ 
																		&= -\textbf{i}(\partial_x^2 + \partial_y^2) \log(\lambda)\, dx \wedge dy \\
																		&= 2 \textbf{i} K_J \rho.
		\end{align*}
This shows
		\begin{align} \label{geoJQeq3} K_J\rho = -\frac{\textbf{i}}{2} \bar{\partial} \partial \log(h) .\end{align}
Plugging (\ref{geoJQeq2}) and (\ref{geoJQeq3}) into (\ref{hk.Meq1}) and using $\bar{\partial} \sigma = 0$ yields 
\begin{align*} 
	 \underline{\mu}_1(J,\sigma) 
			&= -\textbf{i} \frac{\bar{\partial} h \wedge \partial h}{2h \sqrt{1-h}} + \textbf{i} \sqrt{1 -h} \bar{\partial} \partial \log(h) - 2\textbf{i} \bar{\partial} \partial \sqrt{1-h} + 2c\rho \\
			&= -\textbf{i} \left[ -\bar{\partial} \sqrt{1-h} \wedge \partial \log(h) + \sqrt{1-h}\bar{\partial} \partial \log(h) - 2 \bar{\partial} \partial \sqrt{1-h} \right] + 2c\rho\\
					&= \textbf{i} \bar{\partial}\left[ \sqrt{1-h} \partial \log(h) - 2 \partial \sqrt{1-h} \right] + 2c\rho
\end{align*}
where the last equation follows from integration by parts. A primitive for the inner expression is given by
			\begin{align*}
				\sqrt{1-h} \partial \log(h) - 2 \partial \sqrt{1-h} 
									= \frac{1}{h\sqrt{1-h}} \partial h
									= -2 \partial \log(1 + \sqrt{1-h}) + \partial \log(h)
			\end{align*}		
Plugging this into the calculation above and using (\ref{geoJQeq3}) then yields
\begin{align*}
	\underline{\mu}_1(J,\sigma) &= \textbf{i}\bar{\partial} \partial \log(h) - 2\textbf{i} \bar{\partial} \partial \log(1 + \sqrt{1-h}) + 2c\rho\\
															&= -2K_J\rho - \Delta (\log(1 + \sqrt{1-h})) \rho + 2c\rho
	\end{align*}
where $\Delta = d^*d$ is the positive Laplacian which satisfies $(\Delta h)\rho = 2\textbf{i} \bar{\partial} \partial h$.
\end{proof}

\subsubsection{Metric description of the moduli space}

Denote by $\textrm{Met}(\Sigma)$ the space of Riemannian metrics on $\Sigma$. For every $g \in \textrm{Met}(\Sigma)$ there exists a unique complex structure $J = J_g \in \cJ(\Sigma)$ which is compatible with $g$. In the following, we always refer to this complex structure, when discussing holomorphic objects on $(\Sigma, g)$. Define
\begin{align} \label{moduli:QMg}
	\cM_d := \left\{ (g,\sigma)\,\left|\, \begin{array}{c} g\in \textrm{Met}(\Sigma), \, \sigma \in Q(g), \, |\sigma| < 1 \\ \bar{\partial} \sigma = 0,\, K_g + \frac{1}{2}\Delta \log(1 + \sqrt{1- |\sigma|^2}) = c \end{array}\right.\right\}\bigg/ \textrm{Diff}_0(\Sigma).
\end{align}
Proposition \ref{hk.Prop:SimplifyEq} and standard Moser isotopy arguments show that the canonical inclusion $\cM_s \rightarrow \cM_d$ is an isomorphism, where $\cM_s$ is defined by (\ref{moduli:QMb}). The next proposition provides a simpler description of this moduli space.

\begin{Proposition}\label{PropRescaling}
Consider on the space of pairs $(g,\sigma)$ with $|\sigma|_g < 1$ the self-map 
	$$(g,\sigma) \mapsto \left(\left(1 + \sqrt{1 - |\sigma|_g^2}\right)\cdot g, \,\sigma \right).$$ 
This induces a well-defined isomorphism between $\cM_d$ and
\begin{align} \label{moduli:gs}
					\cM := \left\{ (g,\sigma) \,\left|\, \begin{array}{c}g \in \textrm{Met}(\Sigma), \, \sigma \in Q(g),\, |\sigma| < 1 \\  \bar{\partial} \sigma = 0,\, K_g - \frac{c}{2}|\sigma|^2 = \frac{c}{2}\end{array}\right.\right\} \bigg/ \text{Diff}_0(\Sigma)
\end{align}
where $c := 2\pi(2-2\textrm{genus}(\Sigma))/\textrm{vol}(\Sigma,\rho)$.
\end{Proposition}

\begin{proof}
Let $g_0 \in \textrm{Met}(\Sigma)$, let $\sigma \in Q(g_0)$ with $|\sigma|_{g_0} < 1$ and define 
		$$f := 1 + \sqrt{1 - |\sigma|_{g_0}^2},\qquad g := f g_0.$$ 
Then $|\sigma|_{g} = |\sigma|_{g_0}/ f < 1$. For the converse direction, use the relation $(f - 1)^2 = 1 - |\sigma|_{g_0}^2$ to obtain
		\begin{align} \label{finv}
			\frac{1}{f} = \frac{1 + |\sigma|_{g_0}^2/f^2}{2} = \frac{1 + |\sigma|_{g}^2}{2}.
		\end{align}
It follows that one can recover $g_0$ from $(g,\sigma)$ via $g_0 = 2 |\sigma|_{g}/(1 + |\sigma|_{g}^2) \cdot g$. In particular
		$$|\sigma|_{g_0} = \frac{2|\sigma|_g}{1 + |\sigma|_g}.$$
and this shows that $|\sigma|_{g_0} < 1$ if and only if $|\sigma|_g < 1$.
The Gaussian curvature changes under the conformal change to
			$$K_g = \frac{1}{f}\left(K_{g_0} + \frac{1}{2} \Delta_{g_0} \log(f) \right)$$
and (\ref{finv}) then yields
		\begin{align*}
				K_{g} = \frac{1 + |\sigma|_g^2}{2} \left(K_{g_0} + \frac{1}{2} \Delta_{g_0} \log(f) \right).
		\end{align*}
This proves the identification of $\cM$ with $\cM_d$ and the proposition.
\end{proof}

\section{Three geometric models for the moduli space}
We assume throughout this section that $\text{genus}(\Sigma) \geq 2$ and that
			$V := \text{vol}(\Sigma, \rho)  = \pi(2\text{genus}(\Sigma) - 2)$.
The moduli space (\ref{moduli:gs}) constructed in the previous section is then given by
	\begin{align} \label{moduli:HKMJ}
			\mathcal{M} := \left\{ (g,\sigma) \,\left|\, \begin{array}{c} g \in \textrm{Met}(\Sigma), \, \sigma \in Q(g), \, |\sigma| < 1 \\ \bar{\partial} \sigma = 0,\, K_g + |\sigma|_g^2 = -1\end{array}\right.\right\} \bigg/ \text{Diff}_0(\Sigma)
	\end{align}		
It follows from the construction in the previous section that $\cM$ carries a natural hyperkähler structure. The purpose of this section is to establish the following three geometric description of this moduli space proposed by Donaldson \cite{Donaldson:2000}. 

	\begin{enumerate}
		\item $\cM$ embeds as an open neighbourhood of the zero section into the cotangent bundle of Teichmüller space $\mathcal{T}(\Sigma)$. The hyperkähler metric on $\cM$ can then be viewed as the Feix--Kaledin extension of the Weil--Petersson metric on $\mathcal{T}(\Sigma)$.
		
		\item $\cM$ parametrizes the class of almost-Fuchsian hyperbolic $3$-manifolds. These are quasi-Fuchsian $3$-manifolds which possess an incompressible minimal surface with principal curvatures in $(-1,1)$. This surface is then unique and its area provides a Kähler potential for the hyperkähler metric.
		\item $\cM$ embeds as an open subset into the smooth locus of the $\SL(2,\C)$ repre\-sentation variety $\cR_{\textrm{SL}(2,\C)}(\Sigma) := \textrm{Hom}\left(\pi_1(\Sigma), \SL(2,\C) \right) / \SL(2,\C)$. The natural complex structure in this picture corresponds to the second complex structure in the first picture and the Goldman holomorphic symplectic form on $\cR_{\textrm{SL}(2,\C)}(\Sigma)$ (see \cite{Goldman:2004}) restricts to $\underline{\omega}_1 - \textbf{i} \underline{\omega}_3$ along the moduli space $\cM$.
	\end{enumerate}

Following ideas of Uhlenbeck \cite{Uhlenbeck:1983}, we show that there is a natural construction which associates to $[g,\sigma] \in \cM$ a complete hyperbolic $3$-manifold. The isomorphism between $\cM$ and the almost-Fuchsian moduli space is then given in Theorem \ref{Thm:MtoAF}. Next, following Hodge \cite{Hodge:PhD} we describe an explicit embedding of $\cM$ into $\cT(\Sigma) \times \overline{\cT(\Sigma)}$ in Theorem \ref{Thm:MtoAF}. This map is not surjective and both maps are related by the simultaneous uniformization theorem of Bers \cite{Bers:1960}, stated in Theorem \ref{Thm.SimulUnif}. Moreover, one can express every complete hyperbolic $3$-manifold as quotient of hyperbolic space $\H^3$. This gives rise to a natural embedding of the almost-Fuchsian moduli space into $\cR_{\textrm{PSL}(2,\C)}(\Sigma)$ which lifts to $\cR_{\textrm{SL}(2,\C)}(\Sigma)$. Theorem \ref{Thm:MtoR} was outlined by Donaldson \cite{Donaldson:2000} and describes the latter embedding in terms of the non-abelian Hodge correspondence and the theory of Higgs bundles \cite{Hitchin:1987}. Finally, we recall in Theorem \ref{Thm:MtoCotT} a well-known result of Uhlenbeck \cite{Uhlenbeck:1983} which states that the natural map of $\cM$ into $T^*\cT(\Sigma)$ is an embedding and show in Theorem \ref{Thm.FKhk} that the induced metric is indeed the Feix--Kaledin hyperkähler extension of the Weil--Petersson metric.

\subsection{Germs of hyperbolic 3-manifolds and almost-Fuchsian metrics}

Let $g \in \textrm{Met}(\Sigma)$ and $\sigma \in Q(g)$ be a quadratic differential compatible with the conformal structure determined by $g$. The next lemma shows that the equations
		\begin{align} \label{eq:hGerm} K_g + |\sigma|^2 = -1, \qquad \bar{\partial} \sigma = 0 \end{align}
are closely related to the curvature equations of a hyperbolic $3$-manifolds $(Y,g^Y)$ along a minimal surfaces $\Sigma \subset Y$. 

\begin{Lemma} \label{Lemma:QF1}
Let $(Y,g^Y)$ be a Riemannian $3$-manifold and let $(\Sigma,g) \subset (Y, g^Y)$ be an isometrically embedded minimal surface with second fundamental form $h$. Then there exists exists a unique quadratic differential $\sigma \in Q(g)$ with $h = \text{Re}(\sigma)$ and the following is satisfied:

\begin{enumerate}
\item $\sigma$ is holomorphic if and only if
				$$R_z^Y(u,v)w \in T_z\Sigma \qquad \text{for all $z \in \Sigma$ and $u,v,w \in T_z \Sigma$} $$ 
where $R^Y$ denotes the curvature tensor of the ambient manifold $Y$.

\item The intrinsic and extrinsic curvature along $\Sigma$ are related by
		$$\frac{\langle R^Y(u,v)v, u \rangle_{g^Y}}{|u|_{g}^2 |v|_{g}^2 - \langle u,v \rangle_{g}^2} = K_g + |\sigma|_g^2 \qquad \text{for all $u,v \in \textrm{Vect}(\Sigma)$}$$
where $K_g$ denotes the Gaussian curvature of $(\Sigma,g)$.
		
\end{enumerate}
\end{Lemma}

\begin{proof}
Let $z \in \Sigma$ and choose conformal coordinates $(x,y)$ in a neighborhood of $z$. In these coordinates $h$ can be written as
			$$h(x,y) = h_{11}(x,y) dx^2 + h_{22}(x,y) dy^2 + 2 h_{12}(x,y) dx dy.$$
Since $\Sigma \subset Y$ is minimal, its mean curvature vanishes and thus $h_{11} = - h_{22}$. Define a quadratic differential by
			\begin{align} \label{eq:hsigma} \sigma(x,y) = \left(h_{11}(x,y) - \textbf{i} h_{12}(x,y) \right) dz^2. \end{align}
This satisfies $h = \textrm{Re}(\sigma)$ and, since the expression is conformally invariant, it defines a quadratic differential on $\Sigma$.

We prove 1. For $z \in \Sigma$ denote by $\Pi(z) : T_zY \rightarrow T_z\Sigma$ the orthogonal projection determined by $g^Y$. The Mainardi-Codazzi equation yields for $u,v,w \in \textrm{Vect}(\Sigma)$:
	\begin{align*}
			\left(R_z^Y(u,v)w\right)^{\perp} 
			:&= (\mathds{1} - \Pi(z)) R^Y_z(u,v)w \\
			&= (\nabla_u^{\Sigma} h)_z(v,w) - (\nabla^{\Sigma}_v h)_z(u,w) \\
			&= \mathcal{L}_u h(v,w) - \mathcal{L}_v h(u,w) + h([u,v],w) + h(u,\nabla_v^{\Sigma} w) - h(v, \nabla_u^{\Sigma} w) \\
			&= \mathcal{L}_u h(v,w) - \mathcal{L}_v h(u,w) + \left(\mathds{1} - \Pi(z)\right) R_z^{\Sigma}(u,v) w \\
			&= \mathcal{L}_u h(v,w) - \mathcal{L}_v h(u,w).
	\end{align*}
In a conformal chart around $z$ the equation $\left(R^Y(u,v)w\right)^{\perp} = 0$ is thus equivalent to
				$$\partial_1 h_{22}(x,y) = \partial_2 h_{12}(x,y), \qquad \partial_2 h_{11}(x,y) = \partial_1 h_{12}(x,y).$$
Since $h_{11} = -h_{22}$, these are the Cauchy--Riemann equations for the function $h_{11}(x,y) - \textbf{i}h_{12}(x,y)$ and hence equivalent to holomorphicity of $\sigma$.

We prove 2. The Gauss-Codazzi equation yields for $u,v \in \textrm{Vect}(\Sigma)$
	\begin{align} \label{GaussCodazziEqn} \langle R^Y(u,v)v , u \rangle_{g^Y}  = \langle R^{\Sigma}(u,v)v , u \rangle_g - h(u,u)h(v,v) - h(u,v)^2.\end{align}
For a unit vector field $u \in \textrm{Vect}(\Sigma)$ with $|u|_g = 1$ it holds
		$$|\sigma|_g^2 = \textrm{Re}(\sigma(u))^2 + \textrm{Im}(\sigma(u))^2 = - h(u,u) h(Ju, Ju) - h(u,Ju)^2$$
and $\langle R^{\Sigma}(u,Ju)Ju , u \rangle_g = K_g$. Inserting these into (\ref{GaussCodazziEqn}) shows
		$$\langle R^Y(u,Ju)Ju , u \rangle  = K_g + |\sigma|_g^2.$$
Hence $K_g + |\sigma|_g^2$ agrees with the sectional curvature of $T_z\Sigma \subset T_zY$.
\end{proof}

Uhlenbeck \cite{Uhlenbeck:1983} and Taubes \cite{Taubes:2004} observed independently that any solution $(g,\sigma)$ of (\ref{eq:hGerm}) determines a unique hyperbolic metric on a tubular neighborhood $Y_{loc} := \Sigma\times (-\epsilon,\epsilon)$. We use the exponential map to identify $Y_{loc}$ with a subset of the normal bundle of $\Sigma$. The hyperbolic metric has then the product form
		$$g^Y(z,t) = \left(\begin{array}{cc} g_t(z) & 0 \\ 0 & 1 \end{array}\right)$$
where $g_0 = g$ and $\Sigma \times \{0\} \subset Y$ is a minimal surface with second fundamental form $\textrm{Re}(\sigma)$. This is a considerably stronger statement then Lemma \ref{Lemma:QF1} above, which follows from a lengthy calculation using the Bianchi identities.
Moreover, Uhlenbeck \cite{Uhlenbeck:1983} showed under the additional pointwise constraint $|\sigma|_g < 1$ that the hyperbolic metric extends to a complete almost-Fuchsian hyperbolic metric on $Y := \Sigma \times \R$ (see Theorem \ref{Thm:MtoAF} below).

\begin{Definition}[\textbf{Almost-Fuchsian metrics}] \label{Def:AF}
We call a complete hyperbolic metric $g^Y$ on $Y := \Sigma \times \R$ almost-Fuchsian when it has the product shape
		$$g^Y(z,t) = \left(\begin{array}{cc} g_t(z) & 0 \\ 0 & 1 \end{array}\right)$$
with $g_t \in \textrm{Met}(\Sigma)$ and such that $\Sigma \times \{0\} \subset Y$ is a minimal surface with principal curvatures in $(-1,1)$. Denote by $\mathcal{AF}(\Sigma)$ the space of all almost-Fuchsian metrics.
\end{Definition}

An almost-Fuchsian manifold is a hyperbolic $3$-manifolds $Y$ which is isometric to $\Sigma\times \R$ equipped with an almost-Fuchsian metric. The work of Uhlenbeck \cite{Uhlenbeck:1983} shows that a complete hyperbolic $3$-manifold $Y$ is almost-Fuchsian if and only if it admits a minimal and incompressible embedding $\iota: \Sigma \hookrightarrow Y$ with principal curvatures in $(-1,1)$. The next theorem provides an explicit isomorphism between the moduli space $\cM$ and $\mathcal{AF}(\Sigma)/\textrm{Diff}_0(\Sigma)$.

\begin{Theorem} [\textbf{Uhlenbeck \cite{Uhlenbeck:1983}}]\label{Thm:MtoAF}
Let $g \in \textrm{Met}(\Sigma)$ and $\sigma \in Q(g)$ satisfy the equations
		\begin{align} \label{eq:hGermQF} K_g + |\sigma|^2 = -1, \qquad \bar{\partial} \sigma = 0, \qquad |\sigma|_g < 1. \end{align}
For every such pair we define an almost-Fuchsian metric by
		\begin{align} \label{QF.met} 
			g^Y = g^Y_{g,\sigma} = \left(\begin{array}{cc} g \left( \cosh(t) \mathds{1} - \sinh(t) g^{-1} \textrm{Re}(\sigma) \right)^2 & 0 \\ 0 & 1 \end{array}\right).
		\end{align}
This is the unique almost-Fuchsian metric which restricts to $g$ along $\Sigma\times\{0\}$ and such that $\textrm{Re}(\sigma)$ is the second fundamental form of $\Sigma\times\{0\} \subset Y$. In particular, 
	\begin{align} \label{Map:MtoAF}
		\cM \stackrel{\cong}{\longrightarrow} \mathcal{AF}(\Sigma)/\textrm{Diff}_0(\Sigma),\qquad [g,\sigma] \mapsto [g^Y_{g,\sigma}]
	\end{align}
defines an isomorphism of the two moduli spaces.
\end{Theorem}

\begin{proof}
This is a reformulation of Theorem 3.3 in \cite{Uhlenbeck:1983}.
\end{proof}

\begin{Lemma} \label{Lemma:Ymin}
Every almost-Fuchsian manifold $Y = (\Sigma \times \R, g^Y)$ contains a unique closed incompressible minimal surface, which is $\Sigma\times\{0\}$.
\end{Lemma}

\begin{proof}
By Theorem \ref{Thm:MtoAF}, we may assume that $g^Y$ is given by (\ref{QF.met}). A direct calculation shows that the mean curvature along $\Sigma\times\{t\}$ is
			$$H(z,t) = \frac{2 \cosh(t) \sinh(t) (1 + |\sigma(z)|_g^2) }{\cosh(t)^2 + \sinh(t)^2|\sigma(z)|_g^2}.$$
As a vector, this points in positive $t$ direction for $t > 0$ and in negative $t$ direction for $t < 0$. Hence, by the maximum principle, there exists no bounded minimal surface in $Y$ except $\Sigma\times\{0\}$.
\end{proof}

\subsection{A Kähler potential and quasi-Fuchsian manifolds}

This section begins with a brief recollection of well-known properties of hyperbolic space $\H^3$, quasi-Fuchisan groups and the simultaneous uniformization theorem of Bers. Classical references for this material are \cite{Thurston:Book, Bers:1970}.

Next, we describe work of Hodge \cite{Hodge:PhD} which gives rise to an explicit embedding $\cM \hookrightarrow \cT(\Sigma) \times \overline{\cT(\Sigma)}$ which is equivariant with respect to the natural action of the mapping class group and intertwines the second complex structure on $\cM$ with the canonical complex structure on $\cT(\Sigma) \times \overline{\cT(\Sigma)}$. This map is not surjective and its image can be identified with the space of almost-Fuchsian manifolds.

Finally, we show that the functional, which assigns to every almost-Fuchsian manifold the area of its unique minimal surface, is a Kähler potential for the hyperkähler metric with respect to the standard complex structure obtained from $\cT(\Sigma) \times \overline{\cT(\Sigma)}$.

\subsubsection{Hyperbolic space and Kleinian groups}

\paragraph{The upper half plane model.} The upper half plane model of hyperbolic space is $\H^3 := \C\times \R_{>0}$ endowed with the hyperbolic metric
			$$g^{\H^3}_{(z,y)}\left((\hat{z}_1, \hat{y}_1), (\hat{z}_2, \hat{y}_2)\right)  = \frac{\textrm{Re}(\hat{z}_1)\textrm{Re}(\hat{z}_2) + \textrm{Im}(\hat{z}_1)\textrm{Im}(\hat{z}_2)  + \hat{y}_1 \hat{y}_2}{y^2}.$$
Identify $(z,y) \in \H^3$ with the quaternion $\textrm{Re}(z) + \textbf{i}\textrm{Im}(z) + \textbf{j}y + \textbf{k}\cdot 0$ and define
		$$\textrm{SL}(2,\C) \times \H^3 \rightarrow \H^3,\qquad \begin{pmatrix} a & b \\ c& d \end{pmatrix} (z, y) := (a(z + \textbf{j}y) + b) (c(z + \textbf{j}y) + d)^{-1}.$$		
One readily checks that this action is well-defined, preserves the hyperbolic metric, acts transitively on the unit disc bundle, and  identifies the isometry group of $\H^3$ with $\textrm{PSL}(2,\C)$. 
The boundary at infinity $\partial_{\infty} \H^3$ can be identified with $(\C\times\{0\})\cup\{\infty\} \cong S^2$. It follows from the explicit formula above that isometries on $\H^3$ correspond to conformal automorphism of the boundary. The induced action of $\SL(2,\C)$ on the boundary is the standard action given by Möbius transformations.

\paragraph{Kleinian groups.} A Kleinian group is a discrete subgroup $\Gamma < \textrm{PSL}(2,\C)$. 
The limit set $L_{\Gamma} \subset \partial_{\infty} \H^3$ of a Kleinian group $\Gamma$ is defined as follows: Choose $p \in \H^3$ and denote its orbit by $\Gamma(p) \subset \H^3$. Then $L_{\Gamma} \subset \partial_{\infty} \H^3$ is the set of points which can be approximated in the euclidean topology of the closed ball $\H^3 \cup \partial_{\infty} \H^3$ by sequences contained in the orbit $\Gamma(p)$. One readily checks that this definition does not depend on the choice of $p$. The complement $\Omega_{\Gamma} := \partial_{\infty} \H^3 \backslash L_{\Gamma}$ is called the region of discontinuity. This is the largest open subset of the boundary on which $\Gamma$ acts properly and discontinuously. The Ahlfors finiteness theorem asserts that for a finitely generated Kleinian group the quotient $\Omega_{\Gamma}/\Gamma$ is the disjoint union of finitely many Riemann surfaces with finitely many points removed. The hyperbolic manifold $Y := \H^3/\Gamma$ can thus be viewed as hyperbolic cobordism between these surfaces.

\paragraph{Fuchsian and quasi-Fuchsian groups.} A quasi-Fuchsian group is a Kleinian group $\Gamma < \textrm{PSL}(2,\C)$ whose limit set $L_{\Gamma}$ is a Jordan curve and such that both components of its region of discontinuity $\Omega_{\Gamma} =: D_+ \cup D_-$ are preserved by $\Gamma$. For these groups Marsden \cite{Marden:1974} proved that $\H^3/\Gamma$ is diffeomorphic to $(D_+/\Gamma) \times \R$ and $(\H^3 \cup \Omega_{\Gamma})/\Gamma$ is diffeomorphic to $(D_+/\Gamma) \times [0,1]$. A quasi-Fuchsian manifold is a complete hyperbolic $3$-manifold $Y$ which is isometric to $\H^3/\Gamma$ for some quasi-Fuchsian group $\Gamma$. A Fuchsian group is a quasi-Fuchsian group $\Gamma$ whose limit set $L_{\Gamma}$ is a circle. 

Every Fuchsian group is conjugated to a discrete subgroup of $\textrm{PSL}(2,\R)$ and thus determines a hyperbolic surface $(\Sigma,g) := \H^2/\Gamma$. A direct calculation shows that the Fuchsian hyperbolic $3$-manifold $Y:= \H^3/\Gamma$ is isometric to $\Sigma\times \R$ equipped with the metric
\begin{align} \label{eq:Fmet} g^Y(z,t) = \begin{pmatrix}  \cosh(t)^2 g(z) & 0 \\ 0 & 1 \end{pmatrix}\end{align}
where $\Sigma := \H^2/\Gamma$ and $g \in \textrm{Met}(\Sigma)$ is the induced hyperbolic metric. 

It follows from Definition \ref{Def:AF} that every Fuchsian manifold is almost-Fuchsian, and conversely, that every almost-Fuchsian manifold is quasi-isometric to a Fuchsian manifold. In particular, every almost-Fuchsian manifold is quasi-Fuchsian, since every quasi-isometry of $\H^3$ induces a continuous map on its boundary at infinity. The converse is not true: There are examples of quasi-Fuchsian manifolds which admit more then one minimal surface (see \cite{Wang:2012, HuangWang:2013, Hass:2015}) and these cannot be almost-Fuchsian by Lemma \ref{Lemma:Ymin}.

\subsubsection{Simultaneous uniformization}

An \textit{odd coupled pair} is a triple $(\Sigma_{-}, [f], \Sigma_{+})$ consisting of two closed Riemann surfaces $\Sigma_{\pm}$ and the homotopy class of an orientation reversing diffeomorphism $f: \Sigma_- \rightarrow \Sigma_+$. Two odd coupled pairs $(\Sigma_-, [f], \Sigma^+)$ and $(\tilde{\Sigma}_{-}, [\tilde{f}], \tilde{\Sigma}_{+})$ are called equivalent if there exist biholomorphic maps $h_{-}: \tilde{\Sigma}_{-} \rightarrow \Sigma_{-}$ and $h_+: \tilde{\Sigma}_{+} \rightarrow \Sigma_+$ such that $f$ is homotopic to $h_+\circ \tilde{f} \circ h_{-}^{-1}$. 

Now fix a closed oriented Riemann surface $\Sigma$. It is not hard to see that every odd coupled pair of Riemann surfaces of the same genus as $\Sigma$ is isomorphic to a couple of the form
		$$(\Sigma_-, [f], \Sigma^+) \sim  ( (\Sigma, J_-), [\textrm{id}], (\bar{\Sigma}, J_+))$$
for some complex structures $J_- \in \cJ(\Sigma)$ and $J_+ \in \cJ(\bar{\Sigma})$. More precisely, this gives rise to an identification of the space of odd coupled pairs with the quotient $\cT(\Sigma)\times \cT(\bar{\Sigma}) / \textrm{MCG}(\Sigma)$ where $\textrm{MCG}(\Sigma) = \textrm{Diff}_+(\Sigma)/\textrm{Diff}_0(\Sigma)$ denotes the mapping class group.

Every quasi-Fuchsian group $\Gamma < \textrm{PSL}(2,\C)$ gives rise to an odd coupled pair: The Riemann surfaces $\Sigma_{\pm}$ are the two connected components of $\Omega_{\Gamma}/\Gamma$, where $\Omega_{\Gamma}$ denotes the region of discontinuity on the boundary sphere. Moreover, the fundamental groups $\pi_1(\Sigma_{\pm})$ are canonically isomorphic to $\Gamma$ and hence give rise to an isomorphism $\pi_1(\Sigma_-) \rightarrow \pi_1(\Sigma_+)$. This determines a unique homotopy class $[f]$ by the Dehn--Nielsen--Baer Theorem. The simultaneous uniformization theorem of Bers asserts that this constructions provides a bijection between the moduli space of quasi-Fuchsian groups and odd coupled pairs.

\begin{Theorem}[\textbf{Simultaneous uniformization}, Bers \cite{Bers:1960}] \label{Thm.SimulUnif}
Every odd coupled pair $(\Sigma_-, [f], \Sigma_+)$ is equivalent to one which can be represented by a quasi-Fuchsian group $\Gamma < \textrm{PSL}(2,\C)$, which is uniquely determined up to conjugation. 
\end{Theorem}

The quasi-Fuchsian moduli space $\mathcal{QF}(\Sigma) \subset \textrm{Hom}(\pi_1(\Sigma), \textrm{PSL}(2,\C))/\textrm{PSL}(2,\C)$ is defined as the space of injective group homomorphism whose image is a quasi-Fuchsian group, and which are identified up to conjugation. It follows from our discussion that
	$$\mathcal{QF}(\Sigma) \cong \cT(\Sigma)\times \cT(\bar{\Sigma}).$$ 
More precisely, any quasi-Fuchsian representation $\rho: \pi_1(\Sigma) \rightarrow \textrm{PSL}(2,\C)$ induces representations $\rho_{\pm}: \pi_1(\Sigma) \rightarrow \textrm{Aut}(D_{\pm})$, whose push forward under identifications $D_{\pm} \cong \H$ are Fuchsian representations. Note that different identifications give rise to Fuchsian representations which are conjugated in $\textrm{PSL}(2,\R)$. Since Teichmüller space is naturally isomorphic to the space of Fuchsian representations up to conjugations, this gives rise to a well-defined map $\mathcal{QF}(\Sigma) \rightarrow \cT(\Sigma)\times \cT(\bar{\Sigma})$. It follows from the Dehn--Nielsen--Baer Theorem and Theorem \ref{Thm.SimulUnif} that this construction yields indeed a bijective correspondence.

\subsubsection{Embedding into the quasi-Fuchsian moduli space}

We present two maps from $\cM$ into $\cT(\Sigma)\times \overline{\cT(\Sigma)}$. The following proposition is a rather direct consequence of Definition \ref{Def:AF} and makes no claim about holomorphicity.

\begin{Proposition} \label{Prop:AFtoTT_QF}
For an almost-Fuchsian metric
\begin{align*}
			g^Y = g^Y_{g,\sigma} = \left(\begin{array}{cc} g \left( \cosh(t) \mathds{1} - \sinh(t) g^{-1} \textrm{Re}(\sigma) \right)^2 & 0 \\ 0 & 1 \end{array}\right) \in \mathcal{AF}(\Sigma)
\end{align*}
define $g_{\pm}^{\infty} := g(1 + |\sigma|_g^2) \mp 2 \textrm{Re}(\sigma)$ and let $J_{+}(g^Y) \in \cJ(\Sigma)$ and $J_{-}(g^Y) \in \cJ(\bar{\Sigma})$ be the unique complex structures compatible with $g_{\pm}^{\infty}$. Then 
\begin{enumerate}
\item $(\Sigma\times\R, g^Y)$ is isomorphic to the quasi-Fuchsian manifold which corresponds to the odd coupled pair $( (\bar{\Sigma} , J_{-}(g^Y)), [\textrm{id}_{\Sigma}], (\Sigma, J_{+}(g_Y)))$. 
\item The map $\cM \cong \mathcal{AF}(\Sigma)/\textrm{Diff}_0(\Sigma) \hookrightarrow \cT(\Sigma) \times \cT(\bar{\Sigma})$ defined by
	\begin{align} \label{Map:MtoTT_QF}
				[g,\sigma] \mapsto \left[J_{+}(g^Y_{g,\sigma}), J_{-}(g^Y_{g,\sigma}) \right]
	\end{align}
is a mapping class group equivariant embedding.
\end{enumerate}
\end{Proposition}

\begin{proof}
For each $z \in \Sigma$ the curve $\gamma_z(t) = (z,t)$ is a geodesic in $Y$ and any point of the boundary at infinity $\Sigma_{\pm}$ can be represented as a geodesic rays $\gamma_z^{\pm}(t): [0,\infty) \rightarrow Y$ defined by $\gamma_z^{\pm}(t) := \gamma_z(\pm t)$. In particular, it follows that the conformal structure on $\Sigma_t := \Sigma\times\{t\} \subset Y$ converges to the conformal structure on $\Sigma_{\pm}$ as $t\rightarrow \pm \infty$. The conformal structure on $\Sigma_t$ can be represented by 
	$$\tilde{g}_t := g \left(\mathds{1} - \tanh(t) g^{-1} \textrm{Re}(\sigma) \right)^2$$
which is conformally equivalent to $g_t$.  The rescaled metrics converge to 
		$$g_{\pm}^{\infty} := g \left(\mathds{1} \mp g^{-1} \textrm{Re}(\sigma) \right)^2 = g (1 + |\sigma|_g^2) \mp 2 \textrm{Re}(\sigma)$$
as $t \rightarrow \pm \infty$. Here we used the relation $(g^{-1} \textrm{Re}(\sigma))^2 = |\sigma|_g^2 \mathds{1}$. This establishes the given formula and the proposition.
\end{proof}

One can understand the map (\ref{Map:MtoTT_QF}) more explicitly on the level of sections. Denote by $X \subset T^*\H$ the unit disc bundle equipped with the hyperkähler structure from Theorem \ref{ThmHK}. By a result of Hodge \cite{Hodge:PhD}, there exists a unique $\textrm{SL}(2,\R)$-equivariant diffeomorphism $\alpha: X \rightarrow \H\times \bar{\H}$ which intertwines the second complex structure on $X$ with $(\textbf{i},-\textbf{i})$ on $\H\times \bar{\H}$ (see Remark \ref{Rmk:XHH}). This is explicitly given by the formula
\begin{align} \label{defn:ALPHA}
		\alpha(x + \textbf{i}y, u + \textbf{i}v) &= \left(x - \frac{y^2 v}{1- yu} + \textbf{i} \frac{y\gamma}{1 - yu},\, x + \frac{y^2 v}{1 + yu} + \textbf{i} \frac{y\gamma}{1 + yu} \right)
\end{align}	
where $\gamma := \sqrt{1 - y^2(u^2 + v^2)}$. Using Lemma \ref{HJ.QLemma1}, one can identify $X$ the fibres of $\mathcal{Q}_1(\Sigma)$, and then obtain a $\textrm{Diff}(\Sigma)$-equivariant bundle map $\alpha: \mathcal{Q}_1(\Sigma) \rightarrow \cJ(\Sigma) \times \overline{\cJ(\Sigma)}$ which descends to a mapping class group equivariant map
\begin{align} \label{Map:MtoTT_Hodge}
		\cM \cong \cM_s \rightarrow \mathcal{T}(\Sigma) \times \overline{\mathcal{T}(\Sigma)}
\end{align}
where $\cM_s$ denotes the moduli space (\ref{moduli:QMb}).

\begin{Proposition}[Hodge \cite{Hodge:PhD}] \label{Prop:HodgeJ2}
The two maps (\ref{Map:MtoTT_QF}) and (\ref{Map:MtoTT_Hodge}) agree with respect to the natural identification $\overline{\mathcal{T}(\Sigma)} \cong \mathcal{T}(\bar{\Sigma})$ induced by the map $J \mapsto -J$. Moreover, the second complex structure on $\cM$ corresponds to $(\hat{J}_1, \hat{J}_2) \mapsto (J_1 \hat{J}_1, -J_2 \hat{J}_2)$ on $\cT(\Sigma) \times \overline{\cT(\Sigma)}$.
\end{Proposition}

\begin{proof}
Let $(J,\sigma) \in \cQ_1(\Sigma)$ and denote by $g := \rho(\cdot,J\cdot)$ the induced Riemannian metric. Choose a holomorphic chart $\phi:U \rightarrow \Sigma$ and
	$$\phi^*J = J_0, \qquad \phi^*\rho = \lambda^2 dx \wedge dy, \qquad \phi^*g = \lambda^2 (dx^2 + dy^2), \qquad \phi^*\sigma = \lambda^2 (u-\textbf{i}v) dz^2$$
for a smooth functions $u,v: U \rightarrow \R$ and $\lambda: U \rightarrow \R_+$. This chart defines a canonical trivialization of the $\SL(2,\R)$-frame bundle given by the frames $\theta_z := \lambda^{-1} d\phi(z)$. With respect to this trivialization corresponds the pair $(\phi^*J, \phi^*\sigma)$ under the isomorphism (\ref{HJ.Qq}) to the section $s^{loc} := (\textbf{i}, u + \textbf{i}v): U \rightarrow X$. By (\ref{defn:ALPHA}) we then have
	$$\alpha(s^{loc}) := \left(\frac{-v}{1-u} + \textbf{i}\frac{\gamma}{1-u} ,  \frac{v}{1+u} + \textbf{i}\frac{\gamma}{1+u} \right).$$
This corresponds to the two complex structures
	 $$J_{+}^{loc} := j\left(\frac{-v}{1-u} + \textbf{i}\frac{\gamma}{1-u} \right) = \frac{1}{\gamma} \left( \begin{array}{cc} v & - (1+u) \\ 1-u & -v \end{array} \right) \in \cJ(\R^2)$$
	 $$J_{-}^{loc} := j\left(\frac{ v}{1+u} + \textbf{i}\frac{\gamma}{1+u} \right) = \frac{1}{\gamma} \left( \begin{array}{cc} -v & 1-u \\ 1+u & v \end{array} \right) \in \cJ(\R^2)$$
where $j: \H \rightarrow \cJ(\R^2)$ is defined (\ref{Hcpx}). These are compatible with the metrics
	$$g_{+} :=  \left(\begin{array}{cc} 0 & 2 \lambda^2\gamma \\ -2 \lambda^2\gamma & 0 \end{array} \right)J_+^{loc} =  2 \lambda^2 \left( \begin{array}{cc} 1-u & -v \\ -v & 1+u \end{array} \right) = 2(\phi^*g - \phi^*\textrm{Re}(\sigma)).$$
	$$g_{-} :=  \left(\begin{array}{cc} 0 & 2 \lambda^2\gamma \\ -2 \lambda^2\gamma & 0 \end{array} \right)J_-^{loc} =  2 \lambda^2 \left( \begin{array}{cc} 1+u & v \\ v & 1-u \end{array} \right) = 2(\phi^*g + \phi^*\textrm{Re}(\sigma)).$$
This shows that the complex structures $(J_+,J_-)$ associated to $(J,\sigma)$ under the maps $\alpha$ are determined by $g_{\pm} := 2(g \mp \textrm{Re}(\sigma))$. Finally, define $\tilde{g}:= (1 + \sqrt{1 - |\sigma|_g^2}) g$. A short calculation shows
	$$g_{\pm} = 2(g \mp \textrm{Re}(\sigma)) = \tilde{g}(1+ |\sigma|_{\tilde{g}}^2) \pm 2 \textrm{Re}(\sigma)$$ 
and hence $J_{\pm}$ agree with the complex structures defined in Proposition \ref{Prop:AFtoTT_QF}.
\end{proof}

\subsubsection{A Kähler potential for the hyperkähler metric}

Consider the area functional on $\cM \cong \mathcal{AF}(\Sigma)/\textrm{Diff}_0(\Sigma)$ which assigns to every almost-Fuchsian manifold the area of its unique closed minimal surface. It follows from Theorem \ref{Thm:MtoAF} and Lemma \ref{Lemma:Ymin}, that it is given by
	\begin{align} \label{def:A1}	A: \cM \cong \mathcal{AF}(\Sigma)/\textrm{Diff}_0(\Sigma) \rightarrow \R,\qquad A([g,\sigma]) := \textrm{vol}(\Sigma, g). \end{align}
The second complex structure on $\cM$ corresponds by Proposition \ref{Prop:HodgeJ2} to the standard complex structure on $\mathcal{AF}(\Sigma)$ obtained from the embedding into $\cT(\Sigma)\times \overline{\cT(\Sigma)}$. The next theorem verifies a remark of Donaldson which says that the area functional (\ref{def:A1}) is a Kähler potential for the hyperkähler metric on $\cM$ with respect to this complex structure. This has been confirmed by direct arguments along $\cT(\Sigma) \subset \cM$ in \cite{GuoZhengWang:2010}.

\begin{Theorem} \label{Thm:Potential}
The area functional (\ref{def:A1}) provides a Kähler potential for the hyperbolic metric. More precisely
	\begin{align} \label{eq:Potential}  2\textbf{i} \bar{\partial}_{J_2} \partial_{J_2} A = \underline{\omega}_2 .\end{align}
\end{Theorem}

\begin{proof}
On the moduli space $\cM_d$, defined by (\ref{moduli:QMg}), the area functional has the shape 
	\begin{align} \label{def:A2}	A_d: \cM_d \rightarrow \R,\qquad A([g,\sigma]) := \int_{\Sigma} \left(1 + \sqrt{1- |\sigma|_{g}^2}\right) d\textrm{vol}_g .\end{align}
This follows from the identification $\cM \cong \cM_d$ in Proposition \ref{PropRescaling}. In particular, on the original moduli space $\cM_s$, defined by (\ref{moduli:QMa}), one has
	\begin{align} \label{def:A3}	A_s: \cM_s \rightarrow \R,\qquad A([J,\sigma]) := \int_{\Sigma} \left(1 + \sqrt{1- |\sigma|_{J}^2}\right) \rho \end{align}
where the norm $|\cdot|_J$ is defined using the metric $\rho(\cdot, J \cdot)$. Consider the $S^1$-action 
	$$S^1\times \cM_s \rightarrow \cM_s, \qquad e^{\textbf{i}t} [g,\sigma] =  [g, e^{\textbf{i}t}\sigma].$$
It follows from Lemma \ref{propfiberMQ} that $A_s$ is a Hamiltonian function on $(\cM_s, \underline{\omega}_1)$ which generates this $S^1$-action. Moreover, this action rotates the second and third symplectic structure satisfying $\cL_{v_A} \underline{\omega}_2 = - \underline{\omega}_3$ and $\cL_{v_A} \underline{\omega}_3 = \underline{\omega}_2$ for the Hamiltonian vector field $v_A$ of $A_s$. The same formal calculation as in Lemma \ref{Lemma:fibreS1} shows that
	$$d A_s(J_2 w) = \underline{\omega}_1(v_A, J_2 w) = \langle J_1 v_A, J_2 w \rangle = \langle J_3 v_A, w \rangle = \underline{\omega}_3(v_A, w)$$
for all $w \in \textrm{Vect}(\cM_s)$ and therefore
	$$2 \textbf{i} \bar{\partial}_{J_2} \partial_{J_2} H = d (dH\circ J_2) = d \iota(v_A) \underline{\omega}_3 = \cL_{v_A} \underline{\omega}_3 = \underline{\omega}_2.$$
This proves (\ref{eq:Potential}) and the theorem.
\end{proof}

\subsection{Embedding into the \texorpdfstring{$\text{SL}(2,\mathbb{C})$}{SL2C} representation variety}

Let $g^Y$ by a hyperbolic metric on $Y := \Sigma \times \R$. A standard fact from differential geometry asserts that the universal cover $\tilde{Y}$ of $Y$ is isometric to hyperbolic space. This follows from direct arguments considering Jacobi fields or more generally from the Cartan--Ambrose--Higgs theorem. Let $\phi: \tilde{Y} \rightarrow \H^3$ be such an isometry. The push-forward of the deck transformation action of $\pi_1(\Sigma)$ on $\tilde{Y}$ yields then a representation $\rho: \pi_1(\Sigma) \rightarrow \textrm{PSL}(2,\C)$. Different choices of the isometry $\phi$ differ by an element of $\textrm{PSL}(2,\C)$ and lead to conjugated representations. We thus obtain a well-defined embedding
	\begin{align}
		\label{Map:MtoR} \cM \cong \mathcal{AF}(\Sigma)/\textrm{Diff}_0(\Sigma) \rightarrow \cR_{\textrm{PSL}(2,\C)}(\Sigma) := \frac{\textrm{Ham}(\pi(\Sigma), \textrm{PSL}(2, \C))}{\textrm{PSL}(2, \C)}.
	\end{align}
The image is an open subset in the smooth locus of the the representation variety $\cR_{\textrm{PSL}(2,\C)}(\Sigma)$. It is a well-known fact that this embedding can be lifted to the $\textrm{SL}(2,\C)$-representation variety $\cR_{\textrm{SL}(2,\C)}(\Sigma)$, see \cite{Culler:1986}. We discuss an explicit construction of this lift using the theory of Higgs bundles below. 

The variety $\cR_{\textrm{SL}(2,\C)}(\Sigma)$ carries a natural holomorphic symplectic structure, see Goldman \cite{Goldman:2004}. A classical result of Bers \cite{Bers:1970} asserts that the restriction of this complex structure to $\mathcal{QF}(\Sigma)$ corresponds to the standard complex structure on $\cT(\Sigma)\times \overline{\cT(\Sigma)}$. In particular, it follows from Proposition \ref{Prop:HodgeJ2} that the second complex structure on $\cM$ corresponds to multiplication by $\textbf{i}$ on $\cR_{\textrm{SL}(2,\C)}(\Sigma)$. Moreover, the holomorphic symplectic form corresponds to $\underline{\omega}_1 - \textbf{i}\underline{\omega}_3$ on $\cM$. This can be seen by noting that both symplectic forms agree with the Weil--Petersson symplectic form along Teichmüller space, which we embed diagonally into the quasi-Fuchisan moduli space using $\alpha$. It then follows from holomorphicity that both forms agree on all of $\cM$. See Hodge \cite{Hodge:PhD} for more details on this.

\begin{Remark}
The quasi-Fuchsian moduli space carries a natural holomorphic symplectic structure which can be expressed in complex Fenchel--Nielsen coordinates and corresponds to the Goldman holomorphic symplectic structure on $\cR_{\textrm{SL}(2,\C)}(\Sigma)$, see \cite{Platis:2001, Goldman:2004}.
\end{Remark}

\subsubsection{Construction of Higgs bundles}
This section describes a construction of Donaldson which associates to every pair $[g,\sigma] \in \cM$ a solution of the Hitchin equations \cite{Hitchin:1987}. This solution can be used to construct a flat $\SL(2,\C)$-connection and the holonomy representation of this connection gives then rise to an alternative description of the embedding of $\cM$ into $\cR_{\textrm{PSL}(2,\C)}(\Sigma)$.\\

Let $g\in \textrm{Met}(\Sigma)$ and $\sigma \in Q(g)$ be given. Choose a complex line bundle $L \rightarrow \Sigma$ with $L^2 = T\Sigma$ and define $E = L \oplus L^{-1}$. The Levi-Civita connection for $g$ induces a unique $U(1)$-connection $a \in \cA(L)$. Then consider the pair
	\begin{align} \label{eq:DonHitchin}
		A = \begin{pmatrix}a & \frac{\bar{\sigma}}{2} \\ -\frac{\sigma}{2} & -a \end{pmatrix}\in \cA(E) \quad \textrm{and} \quad \phi = \frac{1}{2} \left(\begin{array}{cc} 0 & \textbf{1} \\ 0 & 0 \end{array}\right) \in \Omega^{1,0}(\text{End}(E))
	\end{align}
where 
	\begin{gather*}
		\sigma \in Q(J) =  \Omega^{1,0}(L^{-2}) = \Omega^{1,0}(\text{Hom}(L,L^{-1})) \\
		\bar{\sigma} \in \overline{Q(J)} =  \Omega^{0,1}(L^2) = \Omega^{0,1}(\text{Hom}(L^{-1},L)) \\
		\textbf{1} \in \Omega^0(\text{End}(T\Sigma)) = \Omega^{1,0}(L^2) = \Omega^{1,0}(\text{Hom}(L^{-1},L)).
	\end{gather*}
The adjoint section $\phi^*$ is given by
	$$\phi^* = \frac{1}{2}\left(\begin{array}{cc} 0 & 0 \\ \textbf{1}^* & 0 \end{array}\right) \in \Omega^{0,1}(\text{End}(E))$$
where $\textbf{1}^* = 2\textbf{i} \textrm{dvol}_g \in \Omega^2(\Sigma, \C) = \Omega^{0,1}(\Sigma, T^*\Sigma) = \Omega^{0,1}(\Sigma, \textrm{Hom}(L, L^{-1}))$ and we used the sign convention $\Lambda^{1,1}(T^*\Sigma) \cong \Lambda^{0,1}(T^*\Sigma) \otimes \Lambda^{1,0}(T^*\Sigma)$.

\begin{Lemma}\label{Lemma:Hitchin}
Consider the setup described above. The pair $(g,\sigma)$ satisfies (\ref{eq:hGerm}) if and only if $(A,\phi)$ satisfies the Hitchin equations
		\begin{align} \label{eqHitchin} F_A + [\phi \wedge \phi^*] = 0, \qquad \bar{\partial}_A \phi = 0. \end{align}
When these conditions are satisfied, then $B := A + \phi + \phi^*$ is a flat $\textrm{SL}(2,\C)$-connection.

\end{Lemma}

\begin{proof}
The equation $d_A \phi = 0$ is automatically satisfied, since $\textbf{1} \wedge \sigma = 0$ and $\sigma \wedge \textbf{1} = 0$. 
The curvature forms of $a$ and $-a$ on $L$ and $L^{-1}$ are related to the Gaussian curvature of $g$ by $F_a = \frac{1}{2 \textbf{i}} K_g\textrm{vol}_g$ and $F_{-a}= -\frac{\textbf{i}}{2}K_g\textrm{vol}_g$.
It then follows that
	$$F_A = \begin{pmatrix} K_g + |\sigma|_g^2 & \nabla \bar{\sigma} \\ \nabla \sigma & -K_g - |\sigma|_g^2 \end{pmatrix} \frac{\textrm{dvol}_g}{2\textbf{i}}$$
where the covariant derivative $\nabla$ is obtained from the Levi-Civita connection of $g$. Hence
	$$F_A + [\phi \wedge \phi^*] = \begin{pmatrix} K_g + 1 + |\sigma|_g^2 & \nabla \bar{\sigma} \\ \nabla \sigma & - (K_g + 1 + |\sigma|^2) \end{pmatrix} \frac{1}{2\textbf{i}}\textrm{dovl}_g$$
and this proves the first part of the lemma. 

Note that the condition $\bar{\partial}_A \phi = 0$ is equivalent to $d_A \phi = 0$ and then also implies $d_A \phi^* = 0$. Under this assumption, it follows 
		$$F_B = F_A + d_A(\phi + \phi^*) +\frac{1}{2} [(\phi + \phi^*) \wedge (\phi + \phi^*)] = F_A + [\phi \wedge \phi^*]$$		
and $B$ is flat when $F_A + [\phi \wedge \phi^*] = 0$ in addition.
\end{proof}

The holonomy representation $\rho_{A,\phi}: \pi_1(\Sigma) \rightarrow \textrm{SL}(2,\C)$ of the flat connection $B := A + \phi + \phi^*$ is well-defined up to conjugation and therefore Lemma \ref{Lemma:Hitchin} yields an embedding of $\cM$ into $\cR_{\textrm{SL}(2,\C)}(\Sigma)$. 

\begin{Remark} \label{Rmk:Taubes1}
Taubes \cite{Taubes:2004} constructs a flat $\textrm{SO}(3,\C)$-connection from a pair $(g,\sigma)$ which satisfies (\ref{eq:hGerm}). The resulting connection is essentially the same as the one constructed above. The standard way to obtain a flat $\textrm{SO}(3,\C)$-connection from our setting, is by viewing the Lie algebra bundle $\mathfrak{su}(E)$ as real vector bundle of rank $3$. It is thus naturally a $\textrm{SO}(3)$-bundle and its complexification $\mathfrak{sl}(E)$ can be viewed as $\textrm{SO}(3,\C)$-bundle. With this understood, it follows that any flat $\textrm{SL}(2,\C)$-connection on $E$ induces a flat $\textrm{SO}(3,\C)$-connection on $\mathfrak{sl}(E)$.
\end{Remark}

\subsubsection{Higgs bundles and almost-Fuchsian manifolds}

The connection between hyperbolic $3$-manifolds and the Hitchin equations was observed by Donaldson \cite{Donaldson:1987b}. For this consider the model $\H^3 = \SL(2,\C)/\textrm{SU}(2)$ for hyperbolic space. The Riemannian structure is obtained from the $SU(2)$-invariant inner product on $\mathfrak{sl}(2,\C)$
	$$\mathfrak{sl}(2,\C) \times \mathfrak{sl}(2,\C) \rightarrow \R, \qquad \langle \zeta_1, \zeta_2 \rangle := \textrm{tr}(\zeta_1 \zeta_2^* + \zeta_2 \zeta_1^*).$$
This induces a left-invariant metric on $\SL(2,\C)$ which then descends to the hyperbolic metric on the quotient $\SL(2,\C)/\textrm{SU}(2)$. 

Let $P^c$ and $P$ be the $\textrm{SL}(2,\C)$ and $\textrm{SU}(2)$-frame bundle of $E = L \oplus L^{-1}$. Then $B$ induces a flat connection on the $\H^3$-bundle 
	$$P(\H^3) := P^c \times_{\SL(2,\C)} \left(\SL(2,\C)/\textrm{SU}(2) \right) = P^c / \textrm{SU}(2)$$
and the reduction $P \subset P^c$ gives rise to a section $s_{A,\phi} \in \Omega^0(\Sigma, P(\H^3))$. The next theorem asserts that the quotient $Y := \H^3/\rho_{A,\phi}$ is an almost-Fuchsian manifold and that $s_{A,\phi}$ gives rise to a minimal isometric embedding of $\Sigma$ into $Y$. Moreover, the two maps from $\cM$ into the representation variety $\cR_{\textrm{SL}(2,\C)}(\Sigma)$ defined by (\ref{Map:MtoR}) and via the Hitchin equations in the previous section agree.

\begin{Theorem} \label{Thm:MtoR}
Suppose $(g,\sigma)$ satisfies (\ref{eq:hGerm}) and let $(A,\phi)$ be the corresponding solution of the Hitchin equations (see Lemma \ref{Lemma:Hitchin}). Denote by $(\tilde{\Sigma}, \tilde{g}, \tilde{\sigma})$ the universal cover of $\Sigma$ equipped with the lifted Riemannian metric $\tilde{g}$ and quadratic differential $\tilde{\sigma}$. Then the following holds.
\begin{enumerate}
	\item $s_{A,\phi}$ lifts to a $\pi_1(\Sigma)$-equivariant isometric immersion $\tilde{s}_{A,\phi}: (\tilde{\Sigma}, \tilde{g}) \rightarrow \H^3$ and the second fundamental form of $\tilde{s}_{A,\phi}$ is given by $\textrm{Re}(\tilde{\sigma})$.
	
	\item The holonomy representation $\rho_{A,\phi}: \pi_1(\Sigma) \rightarrow \SL(2,\C)$ of $B := A + \phi + \phi^*$ agrees up to conjugation with the image of $[g,\sigma]$ under (\ref{Map:MtoR}). In particular, $Y := \H^3/\rho_B$ is a smooth almost-Fuchsian manifold and $s_{A,\phi}$ defines a minimal isometric embedding $(\Sigma,g) \hookrightarrow Y$ with second fundamental form $\textrm{Re}(\sigma)$.
\end{enumerate}
\end{Theorem}

\begin{proof}
We recall some of the key observations of Donaldson \cite{Donaldson:1987b}: First, the canonical isomorphism
	\begin{align} \label{adPvert} \textrm{\textbf{i}}\textrm{ad}(P) \cong s_{A,\phi}^* (T^{vert} (P(\H^3)) \end{align}
intertwines the connection induced by $A$ on $\textrm{\textbf{i}}\textrm{ad}(P)$ and the connection induced by $B := A + \phi + \phi^*$ and the Levi-Civita connection of $\H^3$ on $s_{A,\phi}^* (T^{vert} (P(\H^3))$. Second, the associated section $s_{A,\phi}$ satisfies
			$$\nabla s_{A,\phi} = (\phi + \phi^*) \in \Omega^1(\Sigma, \textrm{\textbf{i}}\textrm{ad}(P)) \subset \Omega^1(\Sigma, \textrm{End}(E))$$ 
where we identify $\textrm{\textbf{i}}\textrm{ad}(P)$ with the space of self-adjoint endomorphism of $E$. Moreover, the Hitchin equations (\ref{eqHitchin}) yield $d_A^* (\phi + \phi^*) = 0$ and thus $\nabla^* \nabla s_{A,\phi} = 0$. Solutions to the later equation are called twisted harmonic sections -- they are represented in any flat trivialization by harmonic maps into $\H^3$. The standard relation between harmonic maps and minimal surfaces then shows that $s_{A,\phi}$ is a minimal immersion.

We prove the first part, by a calculation in local coordinates. Let $U \subset \Sigma$ be a contractible holomorphic coordinate chart and suppose $g = \lambda^2 (dx^2 + dy^2)$ in these coordinates. This chart provides a trivialization of $T\Sigma = L^2$ over $U$ and we choose compatible trivializations of $L$ and $L^{-1}$. These trivializations are not unitary, and the bundle metric is given by $\lambda\oplus \lambda^{-1}$. In this trivialization, $s_{A,\phi}$ is represented by a map $s: U \rightarrow \H^3$ with derivative
	$$ds(v) = \frac{1}{2} \begin{pmatrix} 0 &  v \\  \lambda^2\bar{v} & 0 \end{pmatrix} \in \textrm{End}(\C^2).$$
In this formula, we view $ds(v)$ is viewed as section of $\textbf{i}\textrm{ad}(P) \subset \textrm{End}(E)$. Next consider the section $\zeta \in \Omega^1(U, \mathfrak{sl}(2,\C))$ defined by
	$$\zeta(v) = \frac{1}{2} \begin{pmatrix} 0 & \lambda v \\ \lambda \bar{v} & 0 \end{pmatrix}$$
This satisfies $ds(v) = L_s \zeta(v)$, where $L_s : \mathfrak{sl}(2,\C) \rightarrow T_s \H^3$ denotes the infinitesimal action. From this it follows that $|ds(v)|^2 = \lambda^2 |v|^2$ and $s$ is an isometric immersion. We calculate in the same chart
	$$\nabla_u (ds(v))  = [A(u) , ds(v)] = \frac{1}{2} \begin{pmatrix} \frac{1}{2}(\bar{\sigma}(u,v) + \sigma(u,v)) & a(u) v + v a(u) \\ - \lambda^2 (a(u) \bar{v} + \bar{v} a(u)) & -\frac{1}{2}(\bar{\sigma}(u,v) + \sigma(u,v)) \end{pmatrix}$$
for vector fields $u,v: U \rightarrow \C$. It follows from the formula for $ds(v)$ above that
		$$\nu(s) := \begin{pmatrix} \frac{1}{2} & 0 \\ 0 & -\frac{1}{2} \end{pmatrix}.$$
corresponds to the unit normal vector field along the image of $s$. In particular, the normal component of $\nabla_u (ds(v))$ is given by $\textrm{Re}(\sigma) \nu(s)$. This shows that the second fundamental form is given by $\textrm{Re}(\sigma)$ and completes the proof of the first part.
		
By Theorem \ref{Thm:MtoAF} there exists a unique quasi-Fuchisan metric $g^Y \in \mathcal{AF}(\Sigma)$ on $Y := \Sigma \times \R$ for which $\Sigma\times\{0\}$ is a minimal surface with induced metric $g$ and second fundamental form $\textrm{Re}(\sigma)$. This lifts to a hyperbolic metric on $\tilde{Y} := \tilde{\Sigma}\times \R$ which is uniquely determined by $\tilde{g}$ and $\tilde{\sigma}$ (see \cite{Uhlenbeck:1983} Theorem 5.1) and
			$$\tilde{Y} \rightarrow \H^3, \qquad (z,t) \mapsto \exp_{\tilde{s}_{A,\phi}(z)} (t\nu(\tilde{s}_{A,\phi}(z)))$$
is a $\pi_1(\Sigma)$-equivariant isometry. This proves the second part and the theorem.
\end{proof}

\subsection{The cotangent bundle of Teichmüller space }
The cotangent bundle of Teichmüller space can be identified with the space
	\begin{align}
			T^*\cT(\Sigma) := \{(J,\sigma)\,|\, J \in \cJ(\Sigma), \, \sigma \in Q(J),\, \bar{\partial}_J \sigma = 0\}/\textrm{Diff}_0(\Sigma).
	\end{align}	
The next theorem shows that the natural map from $\cM$ into $T^*\cT(\Sigma)$ is an embedding. This follows from a standard application of the continuation method and the proof is due to Uhlenbeck (\cite{Uhlenbeck:1983}, Theorem 4.4). We reproduce the proof below for convenience of the reader.

\begin{Theorem}[\textbf{Uhlenbeck} \cite{Uhlenbeck:1983}]\label{Thm:MtoCotT}
Let $\cM$ be the moduli space (\ref{moduli:HKMJ}). For $g \in \textrm{Met}(\Sigma)$ denote by $J_g \in \cJ(\Sigma)$ the unique complex structure compatible with $g$ and the orientation of $\Sigma$. Then
			\begin{align} \label{MTcotEmb} \cM \rightarrow T^*\cT(\Sigma), \qquad [g,\sigma] \mapsto [J_g, \sigma] \end{align}
is a smooth embedding.
\end{Theorem}

\begin{Remark} \label{Rmk:Taubes2}
The theorem does not hold without the restriction $|\sigma|_g < 1$, see \cite{Uhlenbeck:1983, HuangLuci:2012}. Taubes \cite{Taubes:2004} investigated extensions to the larger space
$$\mathcal{H} := \left\{ (g,\sigma) \,\left|\, \begin{array}{c} g \in \textrm{Met}(\Sigma), \, \sigma \in Q(g), \\ \bar{\partial} \sigma = 0,\, K_g + |\sigma|_g^2 = -1\end{array}\right.\right\} \bigg/ \text{Diff}_0(\Sigma)$$
of \textit{minimal hyperbolic germs}. He shows first of all, that $\mathcal{H}$ is a smooth orientable manifold of (real) dimension $12(\textrm{genus}(\Sigma) - 1)$ and conducts a detailed study of the two maps from $\mathcal{H}$ to the space of flat $\textrm{SL}(2,\C)$-connections and into the cotangent bundle of Teichmüller space. He proves that both maps are not proper, that they have identical critical loci and gives a geometric description of the critical points. Moreover, the pullback of the canonical symplectic forms on the cotangent bundle of Teichmüller space and on the space of flat $\textrm{SL}(2,\C)$-connections agree on $\mathcal{H}$. Another way to put this is that there exists a Lagrangian immersion of $\mathcal{H}$ in the product space of the cotangent bundle of Teichmüller space and the smooth locus of the $\textrm{SL}(2,\C)$ representation variety of $\Sigma$.
\end{Remark}

\begin{proof}[Proof of Theorem \ref{Thm:MtoCotT}]
Let $J \in \cJ(\Sigma)$ and $\sigma \in Q(J)$. We show that there exists at most one metric $g$ in the conformal class determined by $J$ with $|\sigma|_g < 1$ and
		\begin{align} \label{KW:eq1}
				K_g + |\sigma|_g^2 = -1
		\end{align}				
By uniformization, there exists a unique hyperbolic metric $g_0 \in \textrm{Met}(\Sigma)$ which is compatible with $J$. Every other metric in the conformal class of $g_0$ has the shape $g = e^{2u}g_0$ for some smooth function $u:\Sigma \rightarrow \R$. \\

\textbf{Step 1:} \textit{$g := e^{2u}g_0 \in \textrm{Met}(\Sigma)$ solves (\ref{KW:eq1}) if and only if $u$ solves
		\begin{align} \label{KW:eq2}
				\Delta_{g_0} u - 1 + e^{2u} + |\sigma|_{g_0}^2e^{-2u} = 0.
		\end{align}
where $\Delta_{g_0} = d^*d$ denotes the positive Laplacian.}\\

The Gaussian curvature changes as $K_g = e^{-2u}\left(\Delta_{g_0} u - 1 \right)$ and the norm of $\sigma$ changes by $|\sigma|_g^2 = |\sigma|_{g_0}^2 e^{-4u}$. Hence
			$$K_g + |\sigma|_g^2 + 1 = e^{-2u}\left(\Delta_{g_0} u - 1 + e^{2u} + |\sigma|_{g_0}^2e^{-2u} \right)$$
and this proves Step 1.			\\

\textbf{Step 2:} \textit{Fix $k \geq 2$ and define $F: W^{k,2}(\Sigma,\R) \rightarrow W^{k-2,2}(\Sigma,\R)$ by 
		\begin{align} \label{proof:TTQ1} F(u) := \Delta_{g_0} u - 1 + e^{2u} + |\sigma|_{g_0}^2e^{-2u} \end{align}
Suppose $|\sigma|_{g} < 1$ pointwise, then $L_u := dF(u): W^{k,2}(\Sigma,\R) \rightarrow W^{k-2,2}(\Sigma,\R)$ is given by
		\begin{align} \label{proof:TTQ2} L_u \xi := \Delta_{g_0} \xi + 2e^{2u}\xi  - 2|\sigma|_{g_0}^2 e^{-2u} \xi. \end{align}
and this is a positive self-adjoint isomorphism.}\\

The formula for the derivative is immediate. We then calculate
\begin{align*}
	\langle L_u \xi, \xi \rangle_{L^2} &= \int_{\Sigma} \left( |d\xi|_{g_0}^2 + 2 e^{2u} \xi^2 - 2 |\sigma|_{g_0}^2 e^{-2u}\xi^2\right) \textrm{dvol}_{g_0} \\
																		 &= \int_{\Sigma} \left( |d\xi|_{g}^2 + 2 \xi^2 - 2 |\sigma|_{g}^2 \xi^2\right) \textrm{dvol}_{g} \\
																		 &= \int_{\Sigma} \left( |d\xi|_{g}^2 + 2 (1 - |\sigma|_{g}^2) \xi^2\right) \textrm{dvol}_{g}
\end{align*}
This is strictly positive for $\xi \neq 0$ and hence $L_u$ is injective. Since $L_u$ is a lower order pertubation of the Laplacian $\Delta_{g_0}$, it is a Fredholm operator of index $0$, and therefore also surjective.\\

\textbf{Step 3:} \textit{Let $g \in \textrm{Met}(\Sigma)$ and $\sigma \in Q(g)$ with $|\sigma|_{g} < 1$ satisfy (\ref{KW:eq1}). Then there exists a unique smooth path $u: [0,1] \rightarrow W^{k,2}(\Sigma,\R)$, $t \mapsto u_t$, such that
			\begin{align} \label{proof:TTQ3} \Delta_{g_0} u_t - 1 + e^{2u_t} + |t \sigma|_{g_0}^2e^{-2u_t} = 0 \end{align}
 for all $t \in [0,1]$ and $g = g_0 e^{2u_1}$.}\\

First, let $0 \leq t_0 < 1$ and suppose that $u_t \in W^{k,2}(\Sigma,\R)$ is a smooth family of functions satisfying (\ref{proof:TTQ3}) for $t \in (t_0,1]$. We claim that
		\begin{align} \label{proof:TTQ4} \partial_t |\sigma|^2_{g_t} \geq 0\qquad \textrm{for all $t \in (t_0,1]$.} \end{align}
Indeed, differentiating the equation yields
		$$L_{u_t} \dot{u}_t + 2t|\sigma|_{g_0}^2 e^{-2{u_t}} = 0$$
where $L_{u_t}$ is a positive elliptic operator by Step 2, provided that $|t \sigma|_{g_t}^2 < 1$. In this case, it follows from the maximums principle that $\dot{u}_t < 0$ and then $\partial_t |\sigma|^2_{g_t} = \partial_t \left(|\sigma|_{g_0}^2 e^{-4u_t}\right) > 0$. Therefore the set of times $t\in (t_0,1]$ for which (\ref{proof:TTQ4}) holds is open, closed and contains $1$. It follows that (\ref{proof:TTQ4}) is satisfied for all $t \in (t_0,1]$

Next, consider $G: W^{k,2}(\Sigma,\R) \times \R \rightarrow W^{k-2, 2}(\Sigma,\R)$ defined by
		$$G(u,t) = \Delta_{g_0} u - 1 + e^{2u} + |t \sigma|_{g_0}^2e^{-2u}.$$
We need to show that there exists a unique family $u_t$ satisfying $G(u_t,t) = 0$ for all $t \in  [0,1]$ and $g = g_0 e^{2u_1}$. By Step 2, we can apply the inverse function theorem at $G(u_t,t)$ if $|t \sigma|^2_{g_t} < 1$ for $g_t := g_0e^{2u_t}$. For $t = 1$ this is satisfied by assumption, and the solution exists on some interval $(t_0, 1]$. Moreover, it follows from (\ref{proof:TTQ4}) that the condition $|t \sigma|^2_{g_t} < 1$ remains satisfied for all $t \in (t_0,1]$. This yields uniqueness of the solution and openness of the maximal existence interval. It remains to show that $u_t$ converges as $t \rightarrow t_0$. The estimate in Step 2, shows that the family of operators $L_{u_t} : W^{2,2}(\Sigma,\R) \rightarrow L^2(\Sigma,\R)$ is uniformly bounded and hence
		$$\dot{u}_t = L_{u_t}^{-1} \left(2t |\sigma|_{g_0}^2 e^{-2u_t}\right), \qquad t \in (t_0,1]$$
is uniformly bounded in $W^{2,2}(\Sigma,\R)$. It follows $u_t$ converges in $W^{2,2}$ as $t \rightarrow t_0$. Repeating the argument inductively, we see that $\dot{u}_t$ is also uniformly bounded in $W^{k,2}(\Sigma,\R)$ and thus the convergence also holds in $W^{k,2}$. \\

\textbf{Step 4:} \textit{The inclusion (\ref{MTcotEmb}) is an embedding.}\\

Let $g \in \textrm{Met}(\Sigma)$ and $\sigma \in Q(g)$ with $|\sigma|_{g} < 1$ satisfy (\ref{KW:eq1}). By Step 3 there exists a unique path $u: [0,1] \rightarrow W^{k,2}(\Sigma,\R)$ satisfying 
	$$K_{g_t} + |t\sigma|_{g_t}^2 = -1, \qquad g_t := g_0 e^{-2u_t}.$$
For $t = 0$, the maximum principle yields that $u_0 \equiv 0$. We may thus recover the metric $g = g_1$ by following the path of solutions defined $G(u_t,t) = 0$. This shows uniqueness of solutions within the conformal class under the constraint $|\sigma|_g < 1$ and this proves the theorem.
\end{proof}

\begin{Theorem}[\textbf{Donaldson} \cite{Donaldson:2000}, \textbf{Hodge} \cite{Hodge:PhD}] \label{Thm.FKhk}
The hyperkähler structure of $\cM$ agrees with the Feix--Kaledin hyperkähler extension of the Weil--Petersson metric on $\cT(\Sigma)$.
\end{Theorem}

\begin{proof}
Consider the model $\cM_s$ defined by (\ref{moduli:QMa}) of the moduli space. It follows directly from its construction that the first complex structure on $\cM_s$ corresponds to the natural complex structure $(\hat{J}, \hat{\sigma}) \mapsto (J\hat{J}, \textbf{i} \hat{\sigma})$ on $T^*\cT(\Sigma)$. Moreover, the hyperkähler metric on the fibre $\cQ_1(\R^2) \subset T^* \cJ(\R^2)$ is $S^1$-invariant and compatible with the natural holomorphic symplectic structure. Hence, after integration, it follows that the hyperkähler structure on $\cM$ is also $S^1$-invariant and compatible with the holomorphic symplectic structure of $T^*\cT(\Sigma)$.

It only remains to show that the hyperkähler metric on $\cM$ restricts to the Weil--Petersson metric on $\cT(\Sigma)$. For this recall that there are two ways to understand tangent vector $\hat{J}_1, \hat{J}_2 \in T_J \cJ(\Sigma)$, namely as sections in $\Omega^0(\Sigma, \textrm{End}(T\Sigma))$ or as $1$-forms in $\Omega_J^{0,1}(\Sigma, T\Sigma)$. These two perspectives are related by the formula
\begin{align} \label{TJ:eqn1}
\frac{1}{2} \textrm{tr}\left(\hat{J}_1 \hat{J}_2 \right) \rho - \frac{\textbf{i}}{2} \textrm{tr}\left(\hat{J}_1 J \hat{J}_2 \right)\rho = -2\textbf{i} h_J\left(\hat{J}_1 \wedge \hat{J}_2 \right)
\end{align}
where we used the left hand side to define the Kähler structure along $\cJ(\Sigma)$. The tangent space of Teichmüller space can be identified with
	$$T_{[J]} \cT(\Sigma) \cong \cH^{0,1}_J(T\Sigma) := \{\hat{J} \in \Omega^{0,1}(\Sigma, T\Sigma)\, |\, \bar{\partial}_J^* \hat{J} = 0 \}.$$
We claim that for every $J$ with $K_J \equiv c$ and $\hat{J} \in \cH^{0,1}_J(T\Sigma)$ the vector $(\hat{J}, 0)$ is tangential to the vanishing locus of the hyperkähler moment map. Indeed, for any $v \in \textrm{Vect}(\Sigma)$ it holds
$$\underline{\omega}_1(\bar{\partial}_J v,  \hat{J}) = - \frac{1}{2} \int_{\Sigma} \textrm{tr}\left( (\bar{\partial}_J v) J \hat{J} \right) \rho = -\langle \bar{\partial}_J v, J \hat{J} \rangle_{L^2} = \langle J v, \bar{\partial}_J^* \hat{J} \rangle_{L^2} = 0.$$
It now follows from the moment map equation that $(\hat{J}, 0)$ is tangential to the vanishing locus of the first moment map. This implies the assertion, since the remaining two moment maps vanish along Teichmüller space. With this understood and (\ref{TJ:eqn1}), it follows that the hyperkähler metric on $\cM_s$ restricts to
	\begin{align*}
	g_{\cT}(\hat{J}_1, \hat{J}_2) + \textbf{i}\omega_{\cT}(\hat{J_1}, \hat{J}_2) = 2 \textbf{i} \int_{\Sigma} h_J\left(\hat{J}_1 \wedge \hat{J}_2 \right)
	\end{align*}
for $J \in \cJ(\Sigma)$ with $K_J \equiv c$ and $\hat{J}_1, \hat{J}_2 \in \cH^{0,1}_J(T\Sigma)$. Since we have specialized to the case where $c = -2$, we can simplify the right hand side to
	\begin{align*}
	g_{\cT}(\hat{J}_1, \hat{J}_2) + \textbf{i}\omega_{\cT}(\hat{J_1}, \hat{J}_2) = \int_{\Sigma} h_{-1}\left(\hat{J}_1 \wedge * \hat{J}_2 \right)
	\end{align*}
where $h_{-1}$ denotes the hermitian metric defined by the unique hyperbolic metric in the conformal class determined by $J$.
	
\end{proof}

\newpage

\bibliographystyle{plain}
\bibliography{references}

\end{document}